\documentclass[english]{amsart}
\usepackage[latin9]{inputenc}
\usepackage{textcomp}
\usepackage{mathrsfs}
\usepackage{amsthm}
\usepackage{amstext}
\usepackage{amssymb}
\usepackage{esint}

\makeatletter

\newcommand{\lyxmathsym}[1]{\ifmmode\begingroup\def\b@ld{bold}
  \text{\ifx\math@version\b@ld\bfseries\fi#1}\endgroup\else#1\fi}

\numberwithin{equation}{section}
\numberwithin{figure}{section}
\theoremstyle{plain}
\newtheorem{thm}{\protect\theoremname}[section]
  \theoremstyle{definition}
  \newtheorem{defn}[thm]{\protect\definitionname}
  \theoremstyle{remark}
  \newtheorem*{rem*}{\protect\remarkname}
  \theoremstyle{plain}
  \newtheorem{prop}[thm]{\protect\propositionname}
  \theoremstyle{plain}
  \newtheorem{lem}[thm]{\protect\lemmaname}
  \theoremstyle{plain}
  \newtheorem*{lem*}{\protect\lemmaname}
  \theoremstyle{plain}
  \newtheorem{cor}[thm]{\protect\corollaryname}
  \newtheorem*{thm*}{\protect\theoremname}
   \newtheorem{rem}[thm]{\protect\remarkname}
   \theoremstyle{plain}

\makeatother

\usepackage{babel}
  \providecommand{\corollaryname}{Corollary}
  \providecommand{\definitionname}{Definition}
  \providecommand{\lemmaname}{Lemma}
  \providecommand{\propositionname}{Proposition}
  \providecommand{\remarkname}{Remark}
\providecommand{\theoremname}{Theorem}

\begin{document}

\title{Strichartz Estimates for Charge Transfer Models}

\author{Gong Chen}

\date{07/20/15}

\email{gc@math.uchicago.edu}
\urladdr{http://www.math.uchicago.edu/\textasciitilde{}gc/}

\address{Department of Mathematics, The University of Chicago, 5734 South
University Avenue, Chicago, IL 60615, U.S.A}



\thanks{The author feels deeply grateful to his advisor Professor Wilhelm
	Schlag for his kind encouragement, discussions, comments and all the support. The author also wants to thank anonymous referees for their very careful reviews}
\begin{abstract}
In this note, we prove Strichartz estimates for scattering states
of scalar charge transfer models in $\mathbb{R}^{3}$. Following 
the idea of Strichartz estimates which based on \cite{CM,RSS},
we also show the energy of the whole evolution is bounded independently
of time without using the phase space method, for example, in \cite{Graf}.
One can easily generalize our arguments to $\mathbb{R}^{n}$ for $n\geq3$.
Finally, in the last section, we discuss the extension of these results
to matrix charge transfer models in $\mathbb{R}^{3}$.
\end{abstract}
\maketitle

\section{Introduction}

In this note, following the work of \cite{RSS,Cai}, charge transfer
models for Schr\"odinger equations in $\mathbb{R}^{3}$ will be considered.
We study the time-dependent charge transfer Hamiltonian 
\begin{equation}
H(t)=-\frac{1}{2}\Delta+\sum_{j=1}^{m}V_{j}(x-\vec{v}_{j}t)\label{eq:11}
\end{equation}
with rapidly decaying smooth potentials $V_{j}(x)$, say, exponentially
decaying and a set of mutually non-parallel constant velocities $\vec{v}_{j}$.
Strichartz estimates for the evolution 
\begin{equation}
\frac{1}{i}\partial_{t}\psi+H(t)\psi=0\label{eq:12}
\end{equation}
associated with a charge transfer Hamiltonian $H(t)$ will be proved.

The starting point is the well-known $L^{p}$ estimates for the free
Schr\"odinger equation ($H_{0}=-\frac{1}{2}\Delta$) on $\mathbb{R}^{n}$
:

\begin{equation}
\left\Vert e^{iH_{0}t}f\right\Vert _{L^{p}}\leq C_{p}\left|t\right|^{-n(\frac{1}{2}-\frac{1}{p})}\left\Vert f\right\Vert _{L^{p'}},\label{eq:13}
\end{equation}
where $2\leq p\leq\infty$, $\frac{1}{p}+\frac{1}{p'}=1$.

To analyze the dispersive estimate of linear Schr\"odinger equations
with potentials, we consider the dispersive estimates of the Schr\"odinger
flow 
\begin{equation}
e^{itH}P_{c},\,\,\,\, H=-\frac{1}{2}\Delta+V\label{eq:15}
\end{equation}
on $\mathbb{R}^{n}$, where $P_{c}$ is the projection onto the continuous
spectrum of $H$. For Schr\"odinger equations with potentials, there may be bound states, i.e., $L^{2}$ eigenfunctions
of $H$. Under the evolution $e^{itH}$, such bound states are merely
multiplied by oscillating factors and thus do not disperse. So we
need to project away any bound state.  $V$ is a real-valued potential that is assumed
to satisfy some decay condition at infinity. This decay is typically
expressed in terms of the point-wise decay $|V(x)|\text{\ensuremath{\le}}C\left\langle x\right\rangle ^{-\beta}$,
for all $x\in\mathbb{R}^{n}$ and for some $\beta>0$. We use the
notation $\langle x\rangle=\left(1+|x|^{2}\right)^{\frac{1}{2}}$.
Occasionally, we will use an integrability condition $V\in L^{p}\left(\mathbb{R}^{n}\right)$
(or a weighted variant of it) instead of a point-wise condition.  These decay conditions will
also be such that $H$ is asymptotically complete:
\[
L^{2}(\mathbb{R}^{n})=L_{p.p.}^{2}\left(\mathbb{R}^{n}\right)\oplus L_{a.c.}^{2}\left(\mathbb{R}^{n}\right)
\]
where the spaces on the right-hand side refer to the span of all eigenfunctions,
and the absolutely continuous subspace, respectively.

The dispersive estimate for the linear Schr\"odinger equations with
potentials, which we will be most concerned with is of the form 

\begin{equation}
\sup_{t\neq0}|t|^{\frac{n}{2}}\left\Vert e^{itH}P_{c}f\right\Vert _{L_{x}^{\infty}}\leq C\left\Vert f\right\Vert _{L_{x}^{1}},\,\,\,\,\,\forall f\in L_{x}^{1}\left(\mathbb{R}^{n}\right)\cap L_{x}^{2}\left(\mathbb{R}^{n}\right).\label{eq:16}
\end{equation}
 Interpolating with the $L^{2}$ bound $\left\Vert e^{itH}P_{c}f\right\Vert _{L_{x}^{2}}\leq C\left\Vert f\right\Vert _{L_{x}^{2}}$,
we get 

\begin{equation}
\sup_{t\neq0}|t|^{n(\frac{1}{2}-\frac{1}{p})}\left\Vert e^{itH}P_{c}f\right\Vert _{L_{x}^{p'}}\leq C\left\Vert f\right\Vert _{L_{x}^{p}}\,\,\,\,\,\forall f\in L_{x}^{1}(\mathbb{R}^{n})\cap L_{x}^{2}(\mathbb{R}^{n}),\label{eq:17}
\end{equation}
where $1\leq p\leq2$. 

It is well-known that via a $T^{*}T$ argument the dispersive estimate
\eqref{eq:16} gives rise to the class of Strichartz estimates

\begin{equation}
\left\Vert e^{itH}P_{c}f\right\Vert _{L_{t}^{q}L_{x}^{p}}\lesssim\left\Vert f\right\Vert _{L^{2}}\label{eq:ordstri}
\end{equation}
 for all $\frac{2}{q}+\frac{n}{p}=\frac{n}{2}$. The endpoint $q=2$
holds for $n\geq3$ but it is not captured by this approach, see \cite{KT}.

Roughly speaking, Strichartz estimates can be regarded as smoothing
effects in $L_{x}^{p}$ spaces. For example, when we consider the
free Schr\"odinger equation, compared with the trivial conservation of
$L^{2}$ norm of the solution, in Strichartz estimates one gains space
integrability from $p=2$ to $p>2$, but one loses time integrability
from $q=\infty$ to $q<\infty$. To be more precise, we can take a function $g\in L^{2}$ but $g\notin L_{x}^{p}$ for $p>2$.
Then we take $f=e^{\frac{1}{2}it_{0}\Delta}g$ as
the initial data for the free linear Schr\"odinger equation, then
we can see at $t=t_{0}$, $e^{-i\frac{1}{2}t_{0}\Delta}f\notin L_{x}^{p}$.
So without integration or average on time, there is no hope to get
$L_{x}^{p}$ estimate for all the time for general $L^{2}$ initial
data.  Strichartz estimates are crucial for the study of long-time behavior of associated nonlinear models.

For the results and historical progress of dispersive estimates and smoothing effects of Schr\"odinger operators, one can find
further details and references in \cite{Sch}.

There are extra difficulties for Schr\"odinger equations with time-dependent potentials. For example, given a general time-dependent potential $V(x,t)$, it is not clear how to introduce an analog of bound states and the spectral projection. And the evolution of equation might not satisfy group properties any more. In this paper, we focus on a particular case of time-dependent potentials, i.e. the charge transfer models in $\mathbb{R}^{3}$.  

Firstly, we consider the scalar model in the following sense:
\begin{defn}
\label{Charge} By a charge transfer model we mean a Schr\"odinger
equation 
\begin{equation}
\frac{1}{i}\partial_{t}\psi-\frac{1}{2}\Delta\psi+\sum_{j=1}^{m}V_{j}(x-\vec{v}_{j}t)\psi=0,\label{eq:18}
\end{equation}
\[
\psi|_{t=0}=\psi_{0},\,\, x\in\mathbb{R}^{3}.
\]

where $\vec{v}_{j}$'s are distinct vectors in $\mathbb{R}^{3}$ ,
and the real potentials $V_{k}$ are such that for every $1\leq k\leq m$ 

1) $V_{k}$ is time-independent and decays exponentially (or has compact
support)

2) $0$ is neither an eigenvalue nor a resonance of the operators
\begin{equation}
H_{k}=\lyxmathsym{\textminus}\frac{1}{2}\Delta+V_{k}(x).\label{eq:19}
\end{equation}

\end{defn}
Recall that $\psi$ is a resonance at $0$ if it is a distributional
solution of the equation $H_{k}\psi=0$ which belongs to the space $\left\{ f:\,\left\langle x\right\rangle ^{-\sigma}f\in L^{2}\right\} $ for
any $\sigma>\frac{1}{2}$, but not for $\sigma=0$.

To simplify our argument, we discuss when $m=2$ case with $V_{1}$
is stationary and $V_{2}$ moves along $\overrightarrow{e_{1}}$ with
the unit speed. It is easy to see our arguments work for general cases.
\begin{rem*}
The assumptions are always assumed when we want to prove dispersive
estimate and Strichartz estimates, e.g, \cite{JSS,Sch,Ya,RSS,Cai}.
The decay required of the potentials is not optimal but merely for convenience. 
\end{rem*}
An indispensable tool in the study of charge transfer models are the Galilei transformations

\begin{equation}
\mathit{\mathfrak{g}}_{\vec{v},y}(t)=e^{i\frac{\left|\vec{v}\right|^{2}}{2}t}e^{ix\cdot\vec{v}}e^{-i\left(y+\vec{v}t\right)\cdot\vec{p}},\label{eq:110}
\end{equation}
cf. \cite{Graf,Cai,RSS}, where $\vec{p}=-i\vec{\nabla}$. They are
the quantum analogues of the classical Galilei transforms 
\begin{equation}
x\mapsto x-t\vec{v}-y,\,\,\,\,\,\,\vec{p}\mapsto\vec{p}-\vec{v}.\label{eq:111}
\end{equation}
To see this, we take a Schwartz function $f$ such that $f$ and $\hat{f}$
are centered around the origin, then $\mathit{\mathfrak{g}}_{\vec{v},y}(t)f$
is centered around $t\vec{v}+y$, and $\widehat{\mathit{\mathfrak{g}}_{\vec{v},y}(t)f}$
is centered around $\vec{v}$. The Galilei transformations have a
very important conjugacy property:

\begin{equation}
\mathit{\mathfrak{g}}_{\vec{v},y}(t)e^{it\frac{\Delta}{2}}=e^{it\frac{\Delta}{2}}\mathit{\mathfrak{g}}_{\vec{v},y}(0).\label{eq:112}
\end{equation}
Moreover, notice that with $H=-\frac{1}{2}\Delta+V$, then 
\begin{equation}
\psi(t):=\mathit{\mathfrak{g}}_{\vec{v},y}(t)e^{-itH}\mathit{\mathfrak{g}}_{\vec{v},y}(0)\psi_{0},\,\,\,\mathit{\mathfrak{g}}_{\vec{v},y}(t)=e^{-iy\vec{v}}\mathit{\mathfrak{g}}_{-\vec{v},-y}(t),\label{eq:114}
\end{equation}
solves 
\begin{equation}
\frac{1}{i}\partial_{t}\psi-\frac{1}{2}\Delta\psi+V\left(\cdot-t\vec{v}-y\right)\psi=0,\,\,\,\psi|_{t=0}=\psi_{0}.\label{eq:115}
\end{equation}
Another important property of the Galilei transformations is that
$\mathit{\mathfrak{g}}_{\vec{v},y}(t)$ are isometries in all $L^{p}$
spaces. Finally, in our case, as discussed above, we always assume
$y=0$. To simplify our notations, we write $\mathit{\mathfrak{g}}_{\vec{v}}(t):=\mathit{\mathfrak{g}}_{\vec{v},0}(t)$
and notice that $\mathit{\mathfrak{g}}_{\vec{e_{1}}}(t)^{-1}=\mathit{\mathfrak{g}}_{-\vec{e_{1}}}(t)$.

We recall some consequences from \cite{RSS,Cai}.
Again, we consider
\begin{equation}
\frac{1}{i}\partial_{t}\psi-\frac{1}{2}\Delta\psi+V_{1}\psi+V_{2}\left(\cdot-t\vec{e_{1}}\right)\psi=0,\,\,\,\psi|_{t=0}=\psi_{0},\label{eq:1111}
\end{equation}
 with $V_{1}$ and $V_{2}$ decaying rapidly. Let $w_{1},\,\ldots,\, w_{m}$
and $u_{1},\,\ldots,\, u_{\ell}$ be the normalized bound states of
$H_{1}$ and $H_{2}$ associated to the negative eigenvalues $\lambda_{1},\,\ldots,\,\lambda_{m}$
and $\mu_{1},\,\ldots,\,\mu_{\ell}$ respectively (notice that by
our assumptions, $0$ is not an eigenvalue). 

Following the notations in \cite{RSS}, we denote by $P_{b}\left(H_{1}\right)$
and $P_{b}\left(H_{2}\right)$ the projections onto the the bound states
of $H_{1}$ and $H_{2}$, respectively, and let $P_{c}\left(H_{i}\right)=Id-P_{b}\left(H_{i}\right),\, i=1,2$.
To be more explicit, we have 
\begin{equation}
P_{b}\left(H_{1}\right)=\sum_{i=1}^{m}\left\langle \cdot,w_{i}\right\rangle w_{i},\,\,\,\,\, P_{b}\left(H_{2}\right)=\sum_{j=1}^{\ell}\left\langle \cdot,u_{j}\right\rangle u_{j}.\label{eq:116}
\end{equation}
It is well-known, from the standard case with stationary potentials that
we need to project away from bound states as we discussed at the very
beginning. Here following \cite{RSS}, we recall the analogous condition
in our case.
\begin{defn}
\label{SP}Let $U(t,0)\psi_{0}=\psi\left(t,x\right)$  be the solution
of equation \eqref{eq:1111}. We say that $\psi_{0}$ or $\psi\left(x,t\right)$
is asymptotically orthogonal to the bound states of $H_{1}$ and $H_{2}$
if 
\begin{equation}
\left\Vert P_{b}\left(H_{1}\right)U(t,0)\psi_{0}\right\Vert _{L^{2}}+\left\Vert P_{b}\left(H_{2},t\right)U(t,0)\psi_{0}\right\Vert _{L^{2}}\rightarrow0,\,\,\, t\rightarrow\pm\infty.\label{eq:117}
\end{equation}
Here 
\[
P_{b}\left(H_{2},t\right):=\mathit{\mathfrak{g}}_{\vec{e_{1}}}(t)^{-1}P_{b}\left(H_{2}\right)\mathit{\mathfrak{g}}_{\vec{e_{1}}}(t),\,\,\,\,\forall t.
\]
It is clear that all $\psi_{0}$ that satisfy \eqref{eq:117} form a closed subspace of $L^{2}\left(\mathbb{R}^{n}\right)$. We call elements in this subspace scattering states at $t=0$ and denote the subspace by $H_{s}(0)$. We name $H_{s}(0)$ as scattering space at $t=0$. With $H_{s}(0)$, we define $P_{s}(0)$ to be the projection onto $H_{s}(0)$.
\end{defn}
\begin{rem*}
	The subspace above coincides
	with the space of scattering states for the charge transfer problem which appears in
	Graf's asymptotic completeness result \cite{Graf}. We will see more details in Section \ref{sec:prelim}.
\end{rem*}

We now formulate our main results.

\begin{thm}[Strichartz estimates]\label{thm:Strich} Consider the charge transfer model as in Definition
	\ref{Charge} with two potentials in $\mathbb{R}^{3}$ as above. Suppose the initial data $\psi_{0}\in L^{2}\left(\mathbb{R}^{3}\right)$ is asymptotically orthogonal to the bound states of $H_{1}$
	and $H_{2}$ in the sense of Definition \ref{SP}. Then for $\psi(t,x)=U(t,0)\psi_{0}$ and  a Schr\"odinger admissible pair
	$(p,q)$ in $\mathbb{R}^{3}$, i.e., 
	\begin{equation}
		\frac{2}{p}+\frac{3}{q}=\frac{3}{2}
	\end{equation}
	with $2\leq q\leq\infty,\,p\geq2$, we have 
	\begin{equation}
		\left\Vert \psi\right\Vert _{L_{t}^{p}\left([0,\infty),\,L_{x}^{q}\right)}\leq C\left\Vert \psi_{0}\right\Vert _{L_{x}^{2}}.\label{eq:Stri}
	\end{equation}
	\end{thm}

We also have the boundedness of the energy.
\begin{thm}\label{thm:energy}
	Let $\psi_{0}\in H^{1}\left(\mathbb{R}^{3}\right)$ and $\psi(t,x)=U(t,0)\psi_{0}$
	be a solution to \eqref{eq:1111} with the initial data $\psi_{0}$. Then 
	\begin{equation}
	\sup_{t\in\mathbb{R}}\left\Vert U(t,0)\psi_{0}\right\Vert _{H^{1}}\leq C\left\Vert \psi_{0}\right\Vert _{H^{1}}.\label{eq:32}
	\end{equation}
\end{thm}

The paper is organized as follows: In Section \ref{sec:prelim}, we will recall some results from \cite{RSS,Cai}. Then in Section \ref{sec:Strich}, we establish Strichartz estimates
for the evolution that is not associated to the bound states of $H_{j}$
for the scalar charge transfer model. In Section \ref{sec:energy}, we will show the energy
of the whole evolution is bounded independently of time. Finally,  we will generalize our arguments to non-selfadjoint matrix cases in Section \ref{sec:matrix}.

\section{Preliminaries}\label{sec:prelim}
In this section, we formulate the important results from \cite{RSS,Cai} which are crucial for later sections.

 First of all, if the evolution is asymptotically orthogonal to the bound states of $H_{1}$ and $H_{2}$ , we can actually get a decay rate for \[\left\Vert P_{b}\left(H_{1}\right)U(t,0)\psi_{0}\right\Vert _{L^{2}}+\left\Vert P_{b}\left(H_{2},t\right)U(t,0)\psi_{0}\right\Vert _{L^{2}}\rightarrow0.\]
\begin{prop}[\cite{RSS}, Proposition 3.1]\label{prop:(RSS Prop 3.1)} Let $\psi(t,x)$
	be a solution to \eqref{eq:1111} which is asymptotically orthogonal
	to the bound states of $H_{1}$ and $H_{2}$ in the sense of Definition
	\ref{SP}. Then we have the decay rate that
	\begin{equation}
	\left\Vert P_{b}\left(H_{1}\right)U(t,0)\psi_{0}\right\Vert _{L^{2}}+\left\Vert P_{b}\left(H_{2},t\right)U(t,0)\psi_{0}\right\Vert _{L^{2}}\lesssim e^{-\alpha \left|t\right|}\left\Vert \psi_{0}\right\Vert _{L^{2}}\label{eq:118}
	\end{equation}
	for some $\alpha>0$.\end{prop}

As pointed out above, $\forall\psi_{0}\in L^{2}$ such that the asymptotically
orthogonal condition \ref{eq:117} are satisfied, they form a subspace
$H_{s}(0)\mbox{\ensuremath{\subset}}L^{2}$. We can do a more general
time-dependent construction. Denote the evolution starting from $\tau$
to $t$ by $U(t,\tau)$. Similar as our original construction there
is a subspace $H_{s}(\tau)\subset L^{2}$ such that for $\psi\in H_{s}(s)$,
\[
\left\Vert P_{b}\left(H_{1}\right)U(t,\tau)\psi\right\Vert _{L^{2}}+\left\Vert P_{b}\left(H_{2},t\right)U(t,\tau)\psi\right\Vert _{L^{2}}\rightarrow0.
\]
Similarly as above, we can also obtain a decay rate that for some
$\alpha>0$, 
\[
\left\Vert P_{b}\left(H_{1}\right)U(t,\tau)\psi\right\Vert _{L^{2}}+\left\Vert P_{b}\left(H_{2},t\right)U(t,\tau)\psi\right\Vert _{L^{2}}\lesssim e^{-\alpha\left(t-\tau\right)}\left\Vert \psi\right\Vert _{L^{2}},\,\psi\in H_{s}(s).
\]
It is crucial to notice an important property of $H_{s}\mbox{(\ensuremath{\tau})}$. 
\begin{lem}
	\label{lem:projS}Denote $P_{s}(\tau)$ to be the projection onto
	$H_{s}(\tau)$. Then far arbitrary $s,\,\tau\in\mathbb{R}$, 
	\begin{equation}
	P_{s}(s)U(s,\tau)=U(s,\tau)P_{s}(\tau).\label{eq:commute}
	\end{equation}
\end{lem}
\begin{proof}
	Notice that for $\psi\in H_{s}(\tau)$, then $U(s,\tau)\psi\in H_{s}(s)$.
	Since 
	\begin{eqnarray*}
		\left\Vert P_{b}\left(H_{1}\right)U(t,s)U(s,\tau)\psi\right\Vert _{L^{2}}+\left\Vert P_{b}\left(H_{2},t\right)U(t,s)U(s,\tau)\psi\right\Vert _{L^{2}}\\
		=\left\Vert P_{b}\left(H_{1}\right)U(t,\tau)\psi\right\Vert _{L^{2}}+\left\Vert P_{b}\left(H_{2},t\right)U(t,\tau)\psi\right\Vert _{L^{2}}\rightarrow0
	\end{eqnarray*}
	as $t\rightarrow\infty$ by the definition of $H_{s}(\tau).$ Then
	again by the definition of $H_{s}(s)$, it is clear $U(s,\tau)\psi\in H_{s}(s)$.
	Conversely, by symmetry, for $\psi\in H_{s}(s)$, then $U(\tau,s)\psi\in H_{s}(\tau)$.
	Therefore, we have the scattering spaces are invariant under the flow
	$U(s,\tau)$, 
	\begin{equation}
	H_{s}(s)=U(s,\tau)H_{s}(\tau).
	\end{equation}
	Let $\phi\in L^{2}$, then $U(s,\tau)P_{s}(\tau)\phi\in H_{s}(s)$
	by construction. Then 
	\begin{eqnarray*}
		U(s,\tau)P_{s}(\tau)\phi & = & \left(1-P_{s}(s)\right)U(s,\tau)P_{s}(\tau)\phi+P_{s}(s)U(s,\tau)P_{s}(\tau)\phi\\
		& = & P_{s}(s)U(s,\tau)P_{s}(\tau)\phi.
	\end{eqnarray*}
	Similarly, 
	\[
	P_{s}(s)U(s,\tau)\phi=P_{s}(s)U(s,\tau)P_{s}(\tau)\phi.
	\]
	Hence 
	\[
	P_{s}(s)U(s,\tau)=U(s,\tau)P_{s}(\tau),
	\]
	as claimed.\end{proof}

If the evolution is asymptotically orthogonal to the bound states, one also have the usual $L^{1}\rightarrow L^{\infty}$ dispersive estimate. 
\begin{thm}[\cite{RSS},\cite{Cai}]\label{thm:dispersive}
	 Consider the charge transfer model as in Definition
	\ref{Charge} with two potentials as above. Assume $\widehat{V_{1}},\,\widehat{V_{2}}\in L^{1}\left(\mathbb{R}^{n}\right)$.
	Then for any initial data $\psi_{0}\in L^{1}\left(\mathbb{R}^{n}\right)$,
	which is asymptotically orthogonal to the bound states of $H_{1}$
	and $H_{2}$ in the sense of Definition \ref{SP}, one has the decay
	estimate 
	\begin{equation}
	\left\Vert U(t,0)\psi_{0}\right\Vert _{L^{\infty}}\lesssim\left|t\right|^{-\frac{n}{2}}\left\Vert \psi_{0}\right\Vert _{L^{1}}.\label{eq:119}
	\end{equation}
	A similar estimate holds for any number of potentials.
\end{thm}
Note that since the potentials depend on time, Strichartz estimates do not follow from the dispersive estimate and $TT^{*}$ argument.

With the decay estimate \eqref{eq:119}, we obtain the asymptotic completeness
of the charge transfer Hamiltonian:
\begin{thm}[\cite{RSS,Cai}]
	\label{thm:ac} Let $w_{1},\,\ldots,\, w_{m}$ and
	$u_{1},\,\ldots,\, u_{\ell}$ be the normalized bound states of $H_{1}$
	and $H_{2}$ associated to the negative eigenvalues $\lambda_{1},\,\ldots,\,\lambda_{m}$
	and $\mu_{1},\,\ldots,\,\mu_{\ell}$. Then for any initial data $\psi_{0}\in L^{2}\left(\mathbb{R}^{n}\right)$,
	the solution $\psi(t,x)=U(t,0)\psi_{0}$ of the charge transfer model,
	equation \eqref{eq:1111}, can be written in the form 
	\begin{equation}
	\psi(t,x)=U(t,0)\psi_{0}=\sum_{r=1}^{m}A_{r}e^{-i\lambda_{r}t}w_{r}+\sum_{s=1}^{\ell}B_{s}e^{-i\mu_{s}t}\mathfrak{g}_{-\overrightarrow{e_{1}}}(t)u_{s}+e^{-it\frac{\Delta}{2}}\phi_{0}+R(t)\label{eq:120}
	\end{equation}
	with some choice of the constants $A_{r},\, B_{s}$ and the function
	$\phi_{0}$. The remainder term $R(t)$ satisfies the estimate,
	\begin{equation}
	\left\Vert R(t)\right\Vert _{L^{2}}\rightarrow0,\,\, t\rightarrow\pm\infty.\label{eq:121}
	\end{equation}
	
\end{thm}

With the asymptotic completeness of the charge transfer Hamiltonian,
we can construct a time-dependent decomposition of $L^{2}$ with scattering
states and analogous of bound states associated with $H_{1}$ and
$H_{2}$. The construction should be similar as \cite{Graf,Ya1}.
Following the notations in \cite{RSS,Graf}, with the proof of Theorem
\ref{thm:ac}, we know the existence of the following wave operators
in $L^{2}$: for $s\in\mathbb{R}$, 

\begin{equation}
\Omega_{0}^{-}(s)=\lim_{t\rightarrow\infty}U(s,t)e^{-iH_{0}(t-s)}
\end{equation}
\begin{equation}
\Omega_{1}^{-}(s)=\lim_{t\rightarrow\infty}U(s,t)e^{-iH_{1}(t-s)}P_{b}\left(H_{1}\right)
\end{equation}
\begin{equation}
\Omega_{2}^{-}(s)=\lim_{t\rightarrow\infty}U(s,t)\mathit{\mathfrak{g}}_{\vec{e_{1}}}(t)^{-1}e^{-iH_{2}(t-s)}P_{b}\left(H_{2}\right)\mathit{\mathfrak{g}}_{\vec{e_{1}}}(t)
\end{equation}
where limits are taken as strong operator topology. From
\cite{RSS}, the ranges of the above operators has the following relation:

\begin{equation}
L^{2}=\mathrm{Ran}\Omega_{0}^{-}(s)\oplus\mathrm{Ran}\Omega_{1}^{-}(s)\oplus\mathrm{Ran}\Omega_{2}^{-}(s).
\end{equation}
Naturally the above constructions will introduce a time-dependent
decomposition of $L^{2}$ and one can observe that 
\[
\mathrm{Ran}\Omega_{0}^{-}(\tau)=H_{s}(\tau).
\]
We introduce projections $P_{i}(\tau)$ to the projection onto the
range of $\Omega_{i}^{-}(\tau)$, $i=1,2$. Clearly. $\mathrm{Ran}\Omega_{1}^{-}(\tau)$
and $\mathrm{Ran}\Omega_{2}^{-}(\tau)$ are analogous as the spans
bound states associated with $H_{1}$ and $H_{2}$ respectively.
Notice that by construction, one can find a basis for $\Omega_{i}^{-}(\tau)$,
$i=1,2$. With our notations above, $w_{1},\,\ldots,\,w_{m}$ and
$u_{1},\,\ldots,\,u_{\ell}$be the normalized bound states of $H_{1}$
and $H_{2}$ associated to the negative eigenvalues $\lambda_{1},\,\ldots,\,\lambda_{m}$
and $\mu_{1},\,\ldots,\,\mu_{\ell}$ respectively. Then $\left\{ w_{i,\tau}=\Omega_{1}^{-}(\tau)w_{i}\right\} _{i=1}^{m}$
is a basis for $\mathrm{Ran}\Omega_{1}^{-}(\tau)$. Similarly, $\left\{ u_{j,\tau}=\Omega_{2}^{-}(\tau)\mathit{\mathfrak{g}}_{\vec{e_{1}}}(\tau)u_{j}\right\} _{i=1}^{m}$
is a basis for $\mathrm{Ran}\Omega_{2}^{-}(\tau)$. By asymptotic
completeness and intuition, as $\tau\rightarrow\infty$, $\Omega_{1}^{-}(\tau)w_{i}\rightarrow w_{i}$
and $\Omega_{2}^{-}(\tau)\mathit{\mathfrak{g}}(\tau)u_{j}\rightarrow\mathit{\mathfrak{g}}_{\vec{e_{1}}}(\tau)^{-1}u_{j}$.
Actually, following \cite{RSS}, one can actually extract a convergent
rate. We focus on$\Omega_{1}^{-}(\tau)w_{i}\rightarrow w_{i}$ and
for the other case, we just need to apply the same argument after applying
a Galilei transformation.
\begin{prop}
	\label{prop:convrate}For some $\alpha>0$ , 
	\begin{equation}
	\left\Vert \Omega_{1}^{-}(\tau)w_{i}-w_{i}\right\Vert _{L^{2}}\lesssim e^{-\alpha\tau}.\label{eq:convrate}
	\end{equation}
\end{prop}
\begin{proof}
	For $1\leq i\leq m$, with Duhamel's formula 
	\[
	U(s,t)e^{-iH_{1}(t-s)}w_{i}=w_{i}+i\int_{s}^{t}U(s,t)V_{2}\left(\cdot-\vec{e_{1}}\tau\right)e^{-i\lambda_{i}(\tau-s)}w_{i}\,d\tau.
	\]
	It suffices to estimate the $L^{2}$ norm of 
	\begin{equation}
	\int_{s}^{\infty}U(s,t)V_{2}\left(\cdot-\vec{e_{1}}\tau\right)e^{-i\lambda_{i}(\tau-s)}w_{i}\,d\tau.
	\end{equation}
	By Agmon's estimate, 
	\begin{eqnarray}
	\left\Vert \int_{s}^{\infty}U(s,t)V_{2}\left(\cdot-\vec{e_{1}}\tau\right)e^{-i\lambda_{i}(\tau-s)}w_{i}\,d\tau\right\Vert _{L^{2}} & \lesssim & \left\Vert \int_{s}^{\infty}V_{2}\left(\cdot-\vec{e_{1}}\tau\right)w_{i}\,d\tau\right\Vert _{L^{2}}\nonumber \\
	& \lesssim & e^{-\alpha s}.
	\end{eqnarray}
	Therefore 
	\[
	\left\Vert \Omega_{1}^{-}(\tau)w-w_{i}\right\Vert _{L^{2}}\lesssim e^{-\alpha\tau},
	\]
	as claimed.
\end{proof}
\section{Strichartz Estimates}\label{sec:Strich}

In this section, we prove Strichartz estimates for charge transfer
models. The ideas will be based on methods in \cite{CM,RSS}. Certainly, we need to project away from the bound states of $H_{1}$
and the moving bound states associated to $H_{2}(t)$. We will
show certain weighted estimates for the evolution of states in the scattering
space defined in \cite{RSS} and in the sense of Definition \ref{SP}.

Now we formulate the following two estimates when our initial
state is in the scattering space. The first one is:
\begin{lem}
\label{lem:We1}For $\sigma>\frac{3}{2}$ and $t\geq t_{0}$

\begin{equation}
\left\Vert \langle x-x_{0}\rangle^{-\sigma}U(t,t_{0})P_{s}(t_{0})\langle x-x_{1}\rangle^{-\sigma}\right\Vert _{2\rightarrow2}\leq C\frac{1}{\langle t-t_{0}\rangle^{\frac{3}{2}}}\label{eq:we1}
\end{equation}
for all $x_{0}$ and $x_{1}$. 
\end{lem}
Here $2\rightarrow2$ means the norm as an operator from $L^{2}$
to $L^{2}$ and $P_{s}$ defined as the projection onto the scattering
space as above in sense of Definition \ref{SP}. Also as usual, $\langle x\rangle=\left(|x|^{2}+1\right)^{\frac{1}{2}}$.
The second weighted estimate we want to show is the following:
\begin{lem}
\label{lem:We2}For $\sigma>\frac{3}{2}$

\begin{equation}
\int_{0}^{\infty}\left\Vert \left\langle x-x(t)\right\rangle ^{-\sigma}U(t,t_{0})P_{s}(t_{0})u\right\Vert _{L_{x}^{2}}^{2}\,dt\leq C\left\Vert u\right\Vert _{L^{2}}^{2}\label{eq:we2}
\end{equation}
for all $x(t)\in\mathcal{C}\left([0,\infty),\mathbb{R}^{3}\right)$. 
\end{lem}
Heuristically, we can see the above two estimates hold for the evolution
of a free Schr\"odinger equation since a free particle moves towards
infinity. The weights just play roles like indicator functions of
certain finite regions. Then surely, as time evolves, the particle will
leave any of those regions. So we have the decay of the wave function.
In our case, the state in the scattering space will just move asymptotically
like a free particle, so we should expect the above result. The second
estimate is a variant of the above heuristics adjusted to our model since we have moving potentials.

Before we prove Lemma \ref{lem:We1}
and Lemma \ref{lem:We2}, we show how to derive Strichartz estimates
for the charge transfer model based on them.

\begin{proof}[Proof of Theorem \ref{thm:Strich}]
Let $\psi(t)=U(t,0)\psi_{0}$ and by our assumption we have $P_{s}(0)\psi_{0}=\psi_{0}$.
Rewrite the charge transfer model as 
\[
i\psi_{t}+\frac{1}{2}\Delta\psi=V_{1}\psi+V_{2}(\cdot-t\vec{e_{1}})\psi.
\]
Now we apply the endpoint Strichartz estimate \cite{KT} for the free
Schr\"odinger equation, we get for a Schr\"odinger admissible pair
$(p,q)$ in $\mathbb{R}^{3}$, one has 
\[
\left\Vert \psi\right\Vert _{L_{t}^{p}[0,\infty)]L_{x}^{q}}\leq C\left\Vert V_{1}\psi+V_{2}(\cdot-t\vec{e_{1}})\psi\right\Vert _{L_{t}^{2}[0,\infty)L_{x}^{\frac{6}{5}}}+\left\Vert \psi_{0}\right\Vert _{L_{x}^{2}}
\]
Since our potentials decay fast, we can pick $m$ large (in particular
$m>\frac{3}{2}$) such that by H\"older's inequality\textbf{ }we have,
\begin{eqnarray*}
\left\Vert V_{1}\psi+V_{2}(\cdot-t\vec{e_{1}})\psi\right\Vert _{L_{t}^{2}[0,\infty)L_{x}^{\frac{6}{5}}} & \leq & \left\Vert V_{1}\psi\right\Vert _{L_{t}^{2}[0,\infty)L_{x}^{\frac{6}{5}}}+\left\Vert V_{2}(\cdot-t\vec{e_{1}})\psi\right\Vert _{L_{t}^{2}[0,\infty)L_{x}^{\frac{6}{5}}}\\
\left\Vert V_{1}\psi\right\Vert _{L_{t}^{2}[0,\infty)L_{x}^{\frac{6}{5}}} & \leq & C_{V}\left\Vert \left\langle x\right\rangle ^{-m}\psi\right\Vert _{L_{t}^{2}[0,\infty)L_{x}^{2}}\\
\left\Vert V_{2}(\cdot-t\vec{e_{1}})\psi\right\Vert _{L_{t}^{2}[0,\infty)L_{x}^{\frac{6}{5}}} & \leq & C_{V}\left\Vert \left\langle x-t\vec{e_{1}}\right\rangle ^{-m}\psi\right\Vert _{L_{t}^{2}[0,\infty)L_{x}^{2}}
\end{eqnarray*}
Now by our above two claimed estimates Lemma \ref{lem:We1} and Lemma
\ref{lem:We2}, we have 
\[
\left\Vert \left\langle x\right\rangle ^{-m}\psi\right\Vert _{L_{t}^{2}[0,\infty)L_{x}^{2}}\leq C\left\Vert \psi_{0}\right\Vert _{L^{2}},
\]
\[
\left\Vert \left\langle x-t\vec{e_{1}}\right\rangle ^{-m}\psi\right\Vert _{L_{t}^{2}[0,\infty)L_{x}^{2}}\leq C\left\Vert \psi_{0}\right\Vert _{L^{2}}.
\]
Then combine all estimates above, we get 
\[
\left\Vert \psi\right\Vert _{L_{t}^{p}[0,\infty)L_{x}^{q}}\leq C\left\Vert \left\langle x\right\rangle ^{-m}\psi\right\Vert _{L_{t}^{2}[0,\infty)L_{x}^{2}}+C\left\Vert \left\langle x-t\vec{e_{1}}\right\rangle ^{-m}\psi\right\Vert _{L_{t}^{2}[0,\infty)L_{x}^{2}}\leq C\left\Vert \psi_{0}\right\Vert _{L^{2}}.
\]
Therefore, we have the desired Strichartz estimate 
\[
\left\Vert \psi\right\Vert _{L_{t}^{p}\left([0,\infty),\,L_{x}^{q}\right)}\leq C\left\Vert \psi_{0}\right\Vert _{L^{2}}.
\]
as claimed.
\end{proof}
In the next section, we will show as a byproduct of Strichartz estimates, \eqref{eq:Stri}, we can get the energy boundedness of the
whole evolution of the charge transfer model.

\subsection{Proof of Lemmas \ref{lem:We1} and \ref{lem:We2}}
To rigorously show Lemmas \ref{lem:We1} and \ref{lem:We2}
are consistent with our heuristics, we consider the free evolution first. We claim that the first estimate \eqref{eq:we1} holds for
the free Schr\"odinger equation. 
\begin{lem}
	\label{lem:FWe1}For $\sigma>\frac{3}{2}$ 
	\begin{equation}
		\left\Vert \langle x-x_{0}\rangle^{-\sigma}e^{i\frac{\Delta}{2}(t-t_{0})}\langle x-x_{1}\rangle^{-\sigma}\right\Vert _{2\rightarrow2}\leq C\frac{1}{\langle t-t_{0}\rangle^{\frac{3}{2}}}\label{eq:fwe1}
	\end{equation}
	for all $x_{0}$ and $x_{1}$ in $\mathbb{R}^{3}$. \end{lem}
\begin{proof}
	Let $s=t-t_{0}$, then if $|s|\leq1$, clearly by $\left\Vert e^{i\frac{\Delta}{2}s}\right\Vert _{2\rightarrow2}\leq1$
	and $\sigma>0$,
	we can get the desired result. 
	
	If $|s|\geq1$, we apply the dispersive
	estimate for the free evolution. Then by Young's inequality we get 
	\[
	\left\Vert \langle x-x_{0}\rangle^{-\sigma}e^{i\frac{\Delta}{2}(t-t_{0})}\langle x-x_{1}\rangle^{-\sigma}\right\Vert _{2\rightarrow2}\lesssim\left\Vert \langle x\rangle^{-\sigma}\right\Vert _{L^{2}}^{2}\left\Vert e^{is\frac{\Delta}{2}}\right\Vert _{1\rightarrow\infty}\lesssim|s|^{-\frac{3}{2}}.
	\]
	So the desired estimate holds.
\end{proof}
Also the second estimate \eqref{eq:we2} holds for the free Schr\"odinger evolution
by the endpoint Strichartz estimate, estimate \eqref{eq:ordstri}.
\begin{lem}
	\label{lem:Fwe2}For $\sigma>\frac{3}{2}$
	
	\begin{equation}
		\int_{0}^{\infty}\left\Vert \left\langle x-x(t)\right\rangle ^{-\sigma}e^{it\frac{\Delta}{2}}u\right\Vert _{L_{x}^{2}}^{2}\,dt\leq C\left\Vert u\right\Vert _{L^{2}}^{2}.\label{eq:fwe2}
	\end{equation}
\end{lem}
\begin{proof}
	By H\"older's inequality, we have 
	
	\[
	\left\Vert \left\langle x-x(t)\right\rangle ^{-\sigma}e^{it\frac{\Delta}{2}}u\right\Vert _{L_{x}^{2}}\lesssim\left\Vert e^{it\frac{\Delta}{2}}u\right\Vert _{L_{x}^{6}}.
	\]
	Then by the endpoint Strichartz estimate in $\mathbb{R}^{3}$, \cite{KT},
	we have 
	\[
	\left\Vert e^{it\frac{\Delta}{2}}u\right\Vert _{L_{t}^{2}L_{x}^{6}}\leq C\left\Vert u\right\Vert _{L^{2}}.
	\]
	Therefore, we can conclude 
	\[
	\int_{0}^{\infty}\left\Vert \left\langle x-x(t)\right\rangle ^{-\sigma}e^{it\frac{\Delta}{2}}u\right\Vert _{L_{x}^{2}}^{2}\,dt\leq C\left\Vert u\right\Vert _{L^{2}}^{2}.
	\]
The Lemma is proved.	
\end{proof}

Now we show Lemmas \ref{lem:We1} and \ref{lem:We2}
by a bootstrap argument similar to the one in \cite{RSS}. As
usual, the constant $C$ varies from line to line.

First of all, we note the following simple facts:
 Since $P_{s}(t_{0})u$
satisfies following estimates, for $p\geq2$, 

\[
\left\Vert P_{b}\left(H_{1}\right)U(t,t_{0})P_{s}(t_{0})\right\Vert _{L^{2}\rightarrow L^{p}}\lesssim e^{-\alpha(p)\left|t-t_{0}\right|},
\]
and 
\[
\left\Vert P_{b}\left(H_{2},t\right)U(t,t_{0})P_{s}(t_{0})\right\Vert _{L^{2}\rightarrow L^{p}}\lesssim e^{-\beta(p)\left|t-t_{0}\right|}.
\]
Then surely,
\[
\left\Vert \langle x-x_{0}\rangle^{-\sigma}P_{b}\left(H_{1}\right)U(t,t_{0})P_{s}(t_{0})\langle x-x_{1}\rangle^{-\sigma}\right\Vert _{2\rightarrow2}\leq C\frac{1}{\langle t-t_{0}\rangle^{\frac{3}{2}}},
\]

\[
\left\Vert \langle x-x_{0}\rangle^{-\sigma}P_{b}\left(H_{2},t\right)U(t,t_{0})P_{s}(t_{0})\langle x-x_{1}\rangle^{-\sigma}\right\Vert _{2\rightarrow2}\leq C\frac{1}{\langle t-t_{0}\rangle^{\frac{3}{2}}}.
\]
For the second weighted estimate, with some $p\geq2$,

\[
\left\Vert \left\langle x-x(t)\right\rangle ^{-\sigma}P_{b}\left(H_{1}\right)U(t,t_{0})P_{s}(t_{0})u\right\Vert _{L_{x}^{2}}\lesssim e^{-\alpha(p)\left|t-t_{0}\right|}\left\Vert u\right\Vert _{L_{x}^{2}},
\]
\[
\left\Vert \left\langle x-x(t)\right\rangle ^{-\sigma}P_{b}\left(H_{2},t\right)U(t,t_{0})P_{s}(t_{0})u\right\Vert _{L_{x}^{2}}\lesssim e^{-\beta(p)\left|t-t_{0}\right|}\left\Vert u\right\Vert _{L_{x}^{2}}.
\]
So 
\[
\int_{0}^{\infty}\left\Vert \left\langle x-x(t)\right\rangle ^{-\sigma}P_{b}\left(H_{1}\right)U(t,t_{0})P_{s}(t_{0})u\right\Vert _{L_{x}^{2}}dt\lesssim\left\Vert u\right\Vert _{L_{x}^{2}}^{2},
\]
and 
\[
\int_{0}^{\infty}\left\Vert \left\langle x-x(t)\right\rangle ^{-\sigma}P_{b}\left(H_{2},t\right)U(t,t_{0})P_{s}(t_{0})u\right\Vert _{L_{x}^{2}}dt\lesssim\left\Vert u\right\Vert _{L_{x}^{2}}^{2}.
\]
By the Duhamel formula, we write  
\begin{eqnarray}
U(t,t_{0})P_{s}(t_{0}) & = & e^{i\frac{1}{2}\Delta(t-t_{0})}P_{s}(t_{0})+i\int_{t_{0}}^{t}e^{i\frac{1}{2}\Delta(t-s)}V_{1}U(s,t_{0})P_{s}(t_{0})\,ds\nonumber \\
&  & +i\int_{t_{0}}^{t}e^{i\frac{1}{2}\Delta(t-s)}V_{2}(\cdot-s\vec{e_{1}})U(s,t_{0})P_{s}(t_{0})\,ds,\label{eq:duham}
\end{eqnarray}
and let
\begin{equation}
U(t,t_{0})=F+iL+iG\label{eq:decomp1}
\end{equation}
Surely, there is no problem with the free piece $F$ as we discussed
above by Lemmas \ref{lem:FWe1} and \ref{lem:Fwe2}.

Now fix $T$ large enough and apply Gronwall's equality. Then we
can find a large constant $C(T)$ such that 
\begin{equation}
\left\Vert \langle x-x_{0}\rangle^{-\sigma}U(t,t_{0})P_{s}(t_{0})\langle x-x_{1}\rangle^{-\sigma}\right\Vert _{2\rightarrow2}\leq C(T)\frac{1}{\langle t-t_{0}\rangle^{\frac{3}{2}}}\label{eq:boot1}
\end{equation}
\begin{equation}
\int_{0}^{T}\left\Vert \left\langle x-x(t)\right\rangle ^{-\sigma}U(t,t_{0})P_{s}(t_{0})u\right\Vert _{L_{x}^{2}}^{2}\,dt\leq C^{2}(T)\left\Vert u\right\Vert _{L^{2}}^{2}\label{eq:boot2}
\end{equation}
hold for $t\leq T$. 

Next we imitate the bootstrap process in \cite{RSS} and \cite{CM}. Fix a large constant $A$ to be determined later. We
also assume $T-t_{0}\gg A$. As in \cite{RSS}, for $t\leq T$
we consider the decomposition of interval $[t_{0},t]=[t_{0},t_{0}+A]\cup[t_{0}+A,t-A]\cup[t-A,t].$
Set 
\[
L_{1}=\int_{t_{0}}^{t_{0}+A}e^{i\frac{1}{2}\Delta(t-s)}V_{1}U(s,t_{0})P_{s}(t_{0})\,ds
\]
\[
L_{2}=\int_{t_{0}+A}^{t-A}e^{i\frac{1}{2}\Delta(t-s)}V_{1}U(s,t_{0})P_{s}(t_{0})\,ds
\]
\[
L_{3}=\int_{t-A}^{t}e^{i\frac{1}{2}\Delta(t-s)}V_{1}U(s,t_{0})P_{s}(t_{0})\,ds
\]
\[
G_{1}=\int_{t_{0}}^{t_{0}+A}e^{i\frac{1}{2}\Delta(t-s)}V_{2}(\cdot-s\vec{e_{1}})U(s,t_{0})P_{s}(t_{0})\,ds
\]
\[
G_{2}=\int_{t_{0}+A}^{t-A}e^{i\frac{1}{2}\Delta(t-s)}V_{2}(\cdot-s\vec{e_{1}})U(s,t_{0})P_{s}(t_{0})\,ds
\]
\[
G_{3}=\int_{t-A}^{t}e^{i\frac{1}{2}\Delta(t-s)}V_{2}(\cdot-s\vec{e_{1}})U(s,t_{0})P_{s}(t_{0})\,ds.
\]
First, we bound $L_{1}$. With Lemma \ref{lem:FWe1}, we have
\begin{eqnarray*}
\left\Vert \langle x-x_{0}\rangle^{-\sigma}L_{1}\langle x-x_{1}\rangle^{-\sigma}\right\Vert _{2\rightarrow2}\;\;\;\;\;\;\;\;\;\;\;\;\;\;\;\;\;\;\;\;\;\;\;\;\;\;\;\;\;\;\;\;\;\;\;\;\;\;\;\;\;\;\;\\
\leq C\int_{t_{0}}^{t_{0}+A}\frac{1}{\left\langle t-s\right\rangle ^{\frac{3}{2}}}\left\Vert \left\langle x\right\rangle ^{-\sigma}U(s,t_{0})P_{s}(t_{0})\langle x-x_{1}\rangle^{-\sigma}\right\Vert _{2\rightarrow2}ds\\
\leq C(A)\int_{t_{0}}^{t_{0}+A}\frac{1}{\left\langle t-s\right\rangle ^{\frac{3}{2}}}\frac{1}{\left\langle s-t_{0}\right\rangle ^{\frac{3}{2}}}\,ds\;\;\;\;\;\;\;\;\;\;\;\;\;\;\;\;\;\;\;\;\;\;\;\;\\
\leq C(A)\frac{1}{\left\langle t-t_{0}\right\rangle ^{\frac{3}{2}}}.\;\;\;\;\;\;\;\;\;\;\;\;\;\;\;\;\;\;\;\;\;
\end{eqnarray*}
Here we just emphasize that the constant in above estimate does not
depend on $T$. Notice that $G_{1}$ can be bounded similarly as above. 

For the second part, with Lemma \ref{lem:Fwe2},
\begin{eqnarray*}
\int_{0}^{T}\left\Vert \left\langle x-x(t)\right\rangle ^{-\sigma}L_{1}u\right\Vert _{L_{x}^{2}}^{2}\,dt\;\;\;\;\;\;\;\;\;\;\;\;\;\;\;\;\;\;\;\;\;\;\;\;\;\;\;\;\;\;\;\;\;\;\;\;\;\;\;\;\;\;\;\;\;\;\;\;\;\;\;\;\;\;\;\;\;\;\\
\leq C\int_{0}^{T}\left(\int_{t_{0}}^{t_{0}+A}\left\langle t-s\right\rangle ^{-\frac{3}{2}}\left\Vert \left\langle x\right\rangle ^{-\sigma}U(s,t_{0})P_{s}(t_{0})u\right\Vert _{L_{x}^{2}}ds\right)^{2}dt\\
\leq C\left\Vert \left\langle x\right\rangle ^{-\sigma}U(s,t_{0})P_{s}(t_{0})u\right\Vert _{L_{t}^{2}\left((t_{0},t_{0}+A),\, L_{x}^{2}\right)}^{2}\\
\leq C(A)\left\Vert u\right\Vert _{L^{2}}^{2}.\;\;\;\;\;\;\;\;\;\;\;\;\;\;\;\;\;\;\;\;\;\;\;\;\;\;\;\;\;
\end{eqnarray*}
Also $G_{1}$ can be bounded similarly. 

Next, we analyze $L_{2}$. With Lemma \ref{lem:FWe1} and the bootstrap
assumption \eqref{eq:boot1}, 
\begin{eqnarray*}
\left\Vert \langle x-x_{0}\rangle^{-\sigma}L_{2}\langle x-x_{1}\rangle^{-\sigma}\right\Vert _{2\rightarrow2}\;\;\;\;\;\;\;\;\;\;\;\;\;\;\;\;\;\;\;\;\;\;\;\;\;\;\;\;\;\;\;\;\;\;\;\;\;\;\;\;\;\;\;\;\;\;\;\;\;\;\;\\
\leq C\int_{t_{0}+A}^{t-A}\frac{1}{\left\langle t-s\right\rangle ^{\frac{3}{2}}}\left\Vert \left\langle x\right\rangle ^{-\sigma}U(s,t_{0})P_{s}(t_{0})\langle x-x_{1}\rangle^{-\sigma}\right\Vert _{2\rightarrow2}ds\\
\leq C(T)\int_{t_{0}+A}^{t-A}\frac{1}{\left\langle t-s\right\rangle ^{\frac{3}{2}}}\frac{1}{\left\langle s-t_{0}\right\rangle ^{\frac{3}{2}}}\,ds\\
\leq CC(T)A^{-\frac{1}{2}}\left\langle t-t_{0}\right\rangle ^{-\frac{3}{2}}
\end{eqnarray*}
for an absolute constant $C$. 

For the other estimate, with Lemma \ref{lem:Fwe2} and bootstrap assumption
\eqref{eq:boot2}, we conclude that
\begin{eqnarray*}
\int_{0}^{T}\left\Vert \left\langle x-x(t)\right\rangle ^{-\sigma}L_{2}u\right\Vert _{L_{x}^{2}}^{2}dt & \leq & C\int_{0}^{T}\left(\int_{t_{0}+A}^{t-A}\left\langle t-s\right\rangle ^{-\frac{3}{2}}\left\Vert \left\langle x\right\rangle ^{-\sigma}U(s,t_{0})P_{s}(t_{0})u\right\Vert _{L_{x}^{2}}ds\right)^{2}dt\\
 & \leq & h(A)\left\Vert \left\langle x\right\rangle ^{-\sigma}U(s,t_{0})P_{s}(t_{0})u\right\Vert _{L_{t}^{2}\left((0,T),\, L_{x}^{2}\right)}^{2}\\
 & \leq & h(A)C^{2}(T)\left\Vert u\right\Vert _{L^{2}}^{2}
\end{eqnarray*}
where 
\[
h(A)\lesssim A^{-1}
\]
by Young's inequality applied to the convolution 
\[
\int_{t_{0}+A}^{t-A}\left\langle t-s\right\rangle ^{-\frac{3}{2}}\left\Vert \left\langle x\right\rangle ^{-\sigma}U(s,t_{0})P_{s}(t_{0})u\right\Vert _{L_{x}^{2}}ds.
\]
So when $A$ is large, we recapture our bootstrap argument conditions,
i.e., $h(A)C^{2}(T)$ will be a small portion of $C(T)$ provided $A$
is large enough. Similar estimates hold for $G_{2}$. 

It remains to analyze $L_{3}$ and $G_{3}$. We will expand $U$ again. And the following two versions of weighted estimates
for Schr\"odinger equations with rapidly decaying potentials will be used.
\begin{lem}
\label{lem:pwe}For $\sigma>\frac{3}{2}$, and $H_{j}=-\frac{1}{2}\Delta+V_{j}$, where $V_{j}$ satisfies the decay assumption for our charge transfer
Hamiltonian, then we have

\begin{equation}
\left\Vert \langle x-x_{0}\rangle^{-\sigma}e^{iH_{j}(t-t_{0})}P_{c}\left(H_{j}\right)\langle x-x_{1}\rangle^{-\sigma}\right\Vert _{2\rightarrow2}\leq C\frac{1}{\langle t-t_{0}\rangle^{\frac{3}{2}}}\label{eq:pwe1}
\end{equation}
and 
\begin{equation}
\int_{0}^{\infty}\left\Vert \left\langle x-x(t)\right\rangle ^{-\sigma}e^{itH_{j}}P_{c}\left(H_{j}\right)u\right\Vert _{L_{x}^{2}}^{2}dt\leq C\left\Vert u\right\Vert _{L^{2}}^{2},\label{eq:pwe2}
\end{equation}
where $P_{c}\left(H_{j}\right)$ is the projection onto the continuous
spectrum of $H_{j}$.\end{lem}
\begin{proof}
These two estimates follow from the boundedness of wave operators
\cite{Ya} and Lemma \ref{lem:FWe1} and Lemma \ref{lem:Fwe2}. Or one can apply the dispersive estimate and Strichartz estimates for perturbed Schr\"odinger equations.
\end{proof}
Now we analyze
\[
L_{3}=\int_{t-A}^{t}e^{i\frac{1}{2}\Delta(t-s)}V_{1}U(s,t_{0})P_{s}(t_{0})\,ds.
\]
Splitting $L_{3}$ with respect to the spectrum of $H_{1}$, one has
\begin{eqnarray*}
	L_{3} & = & \int_{t-A}^{t}e^{i\frac{1}{2}\Delta(t-s)}V_{1}P_{c}\left(H_{1}\right)U(s,t_{0})P_{s}(t_{0})\,ds\\
	&  & +\int_{t-A}^{t}e^{i\frac{1}{2}\Delta(t-s)}V_{1}P_{b}\left(H_{1}\right)U(s,t_{0})P_{s}(t_{0})\,ds\\
	& = & L_{3,c}+L_{3,b}.
\end{eqnarray*}
Surely, there is no problem with $L_{3,b}$ by the discussion at the very beginning of this section$P_{b}\left(H_{1}\right)U(s,t_{0})P_{s}(t_{0})$ decays exponentially.
 
For $L_{3,c}$, we use the ideas from \cite{RSS} to decompose our evolution into
low velocity and high velocity pieces. For the low velocity piece, we
directly use a commutator argument, non-stationary phase and
the fact the supports of  $V_{1}$ and $V_{2}$ become almost disjoint.
For the high velocity part, we use a version of the Kato smoothing estimate.

Expanding $U$ with respect to $H_{1}$, we can write
\[
U(t,t_{0})=e^{-iH_{1}(t-t_{0})}+i\int_{t_{0}}^{t}e^{-iH_{1}(t-s)}V_{2}(\cdot-s\vec{e_{1}})U(s,t_{0})\,ds.
\]
Then we can write 
\begin{eqnarray*}
	L_{3,c} & = & \int_{t-A}^{t}e^{i\frac{1}{2}\Delta(t-s)}V_{1}P_{c}\left(H_{1}\right)U(s,t_{0})P_{s}(t_{0})\,ds\\
	& = & \int_{t-A}^{t}e^{i\frac{1}{2}\Delta(t-s)}V_{1}P_{c}\left(H_{1}\right)e^{-iH_{1}(s-t_{0})}P_{s}(t_{0})\,ds\\
	&  & +i\int_{t-A}^{t}\int_{t_{0}}^{s}V_{1}P_{c}\left(H_{1}\right)e^{-iH_{1}(s-\tau)}V_{2}(\cdot-\tau\vec{e_{1}})U(\tau,t_{0})P_{s}(t_{0})\,d\tau ds.
\end{eqnarray*}
Consider the decomposition
\[
L_{3,c}=I+iK,
\]
\[
I=\int_{t-A}^{t}e^{i\frac{1}{2}\Delta(t-s)}V_{1}P_{c}\left(H_{1}\right)e^{-iH_{1}(s-t_{0})}P_{s}(t_{0})\,ds,
\]
\[
K=\int_{t-A}^{t}e^{i\frac{1}{2}\Delta(t-s)}V_{1}\int_{t_{0}}^{s}P_{c}\left(H_{1}\right)e^{-iH_{1}(s-\tau)}V_{2}(\cdot-\tau\vec{e_{1}})U(\tau,t_{0})P_{s}(t_{0})\,d\tau ds.
\]
There is no problem with $I$ by similar arguments for the free case
with Lemma \ref{lem:pwe}. 

Next, we decompose $K$ further as follows: 
\[
K=J+S+Z,
\]
\[
S=\int_{t-A}^{t}e^{i\frac{1}{2}\Delta(t-s)}V_{1}\int_{t_{0}}^{t_{0}+B}P_{c}\left(H_{1}\right)e^{-iH_{1}(s-\tau)}V_{2}(\cdot-\tau\vec{e_{1}})U(\tau,t_{0})P_{s}(t_{0})\,d\tau ds,
\]
\[
Z=\int_{t-A}^{t}e^{i\frac{1}{2}\Delta(t-s)}V_{1}\int_{t_{0}+B}^{s-B}P_{c}\left(H_{1}\right)e^{-iH_{1}(s-\tau)}V_{2}(\cdot-\tau\vec{e_{1}})U(\tau,t_{0})P_{s}(t_{0})\,d\tau ds,
\]
\[
J=\int_{t-A}^{t}\int_{s-B}^{s}e^{i\frac{1}{2}\Delta(t-s)}V_{1}P_{c}\left(H_{1}\right)e^{-iH_{1}(s-\tau)}V_{2}(\cdot-\tau\vec{e_{1}})U(\tau,t_{0})P_{s}(t_{0})\,d\tau ds.
\]
For $S$, a similar argument as for $L_{1}$ implies 
\begin{eqnarray*}
\left\Vert \langle x-x_{0}\rangle^{-\sigma}S\langle x-x_{1}\rangle^{-\sigma}\right\Vert _{2\rightarrow2}\;\;\;\;\;\;\;\;\;\;\;\;\;\;\;\;\;\;\;\;\;\;\;\;\;\;\;\;\;\;\;\;\;\;\;\;\;\;\;\;\;\;\;\;\;\;\;\;\;\;\;\;\;\;\;\;\;\;\;\;\;\;\;\;\;\;\;\;\;\;\;\;\\
\lesssim C\int_{t-A}^{t}\frac{1}{\left\langle t-s\right\rangle ^{\frac{3}{2}}}\int_{t_{0}}^{t_{0}+B}\frac{1}{\left\langle s-\tau\right\rangle ^{\frac{3}{2}}}\left\Vert \left\langle x\right\rangle ^{-\sigma}U(s,t_{0})P_{s}(t_{0})\langle x-x_{1}\rangle^{-\sigma}\right\Vert _{2\rightarrow2}\,d\tau ds\\
\lesssim C(B)\int_{t-A}^{t}\int_{t_{0}}^{t_{0}+B}\frac{1}{\left\langle t-s\right\rangle ^{\frac{3}{2}}}\frac{1}{\left\langle s-\tau\right\rangle ^{\frac{3}{2}}}\frac{1}{\left\langle s-t_{0}\right\rangle ^{\frac{3}{2}}}dsd\tau\\
\leq C(A,B)\frac{1}{\left\langle t-t_{0}\right\rangle ^{\frac{3}{2}}}.
\end{eqnarray*}
As usual, the constant $C$ does not depend on $T$. 

For the second piece, we also have
\begin{eqnarray*}
\int_{0}^{T}\left\Vert \left\langle x-x(t)\right\rangle ^{-\sigma}Su\right\Vert _{L_{x}^{2}}^{2}dt\;\;\;\;\;\;\;\;\;\;\;\;\;\;\;\;\;\;\;\;\;\;\;\;\;\;\;\;\;\;\;\;\;\;\;\;\;\;\;\;\;\;\;\;\;\;\;\;\;\;\;\;\;\;\;\;\;\;\;\;\;\;\;\;\;\;\;\;\;\;\;\;\;\;\;\;\;\;\;\;\;\;\;\;\;\;\;\\
\leq C\int_{0}^{T}\left(\int_{t-A}^{t}\left\langle t-s\right\rangle ^{-\frac{3}{2}}\int_{t_{0}}^{t_{0}+B}\left\langle s-\tau\right\rangle ^{-\frac{3}{2}}\left\Vert \left\langle x\right\rangle ^{-\sigma}U(s,t_{0})P_{s}(t_{0})u\right\Vert _{L_{x}^{2}}d\tau ds\right)^{2}dt\\
\lesssim C(B)A\left\Vert \left\langle x\right\rangle ^{-\sigma}U(s,t_{0})P_{s}u\right\Vert _{L_{t}^{2}\left((t_{0},t_{0}+B),\, L_{x}^{2}\right)}^{2}\;\;\;\;\;\;\;\;\;\;\;\;\;\;\;\;\;\;\;\;\;\;\;\;\;\;\;\;\;\\
\lesssim C(A,B)\left\Vert u\right\Vert _{L^{2}}^{2}.\;\;\;\;\;\;\;\;\;\;\;\;\;\;\;\;\;\;\;\;\;\;\;\;\;\;\;\;\;\;\;\;\;\;\;\;\;\;\;\;\;\;\;\;\;\;\;\;\;\;\;\;\;\;\;\;\;\;
\end{eqnarray*}
Next, for $Z$, following a similar argument to $L_{2}$, we obtain
\begin{eqnarray*}
\left\Vert \langle x-x_{0}\rangle^{-\sigma}Z\langle x-x_{1}\rangle^{-\sigma}\right\Vert _{2\rightarrow2}\;\;\;\;\;\;\;\;\;\;\;\;\;\;\;\;\;\;\;\;\;\;\;\;\;\;\;\;\;\;\;\;\;\;\;\;\;\;\;\;\;\;\;\;\;\;\;\;\;\;\;\;\;\;\;\;\;\;\;\;\;\;\;\\
\lesssim\int_{t-A}^{t}\frac{1}{\left\langle t-s\right\rangle ^{\frac{3}{2}}}ds\int_{t_{0}+B}^{s-B}\frac{1}{\left\langle s-\tau\right\rangle ^{\frac{3}{2}}}\left\Vert \left\langle x\right\rangle ^{-\sigma}U(\tau,t_{0})P_{s}(t_{0})\langle x-x_{1}\rangle^{-\sigma}\right\Vert _{2\rightarrow2}d\tau\\
\lesssim C(T)\int_{t-A}^{t}\int_{t_{0}+B}^{s-B}\frac{1}{\left\langle t-s\right\rangle ^{\frac{3}{2}}}\frac{1}{\left\langle s-\tau\right\rangle ^{\frac{3}{2}}}\frac{1}{\left\langle \tau-t_{0}\right\rangle ^{\frac{3}{2}}}\,d\tau ds\;\;\;\;\;\;\;\;\;\;\;\\
\lesssim C(T)AB^{-\frac{1}{2}}\left\langle t-t_{0}\right\rangle ^{-\frac{3}{2}}.\;\;\;\;\;\;\;\;\;\;\;\;\;\;\;\;\;\;\;\;\;\;
\end{eqnarray*}
For the second estimate,
\begin{eqnarray*}
\int_{0}^{T}\left\Vert \left\langle x-x(t)\right\rangle ^{-\sigma}Zu\right\Vert _{L_{x}^{2}}^{2}dt\;\;\;\;\;\;\;\;\;\;\;\;\;\;\;\;\;\;\;\;\;\;\;\;\;\;\;\;\;\;\;\;\;\;\;\;\;\;\;\;\;\;\;\;\;\;\;\;\;\;\;\;\;\;\;\;\;\;\;\;\;\;\;\;\;\;\;\;\;\;\;\;\;\;\;\;\;\;\;\;\\
\lesssim\int_{0}^{T}\left(\int_{t-A}^{t}\left\langle t-s\right\rangle ^{-\frac{3}{2}}\int_{t_{0}+B}^{s-B}\left\langle s-\tau\right\rangle ^{-\frac{3}{2}}\left\Vert \left\langle x\right\rangle ^{-\sigma}U(\tau,t_{0})P_{s}(t_{0})u\right\Vert _{L_{x}^{2}}d\tau ds\right)^{2}dt\\
\lesssim h(B)A\left\Vert \left\langle x\right\rangle ^{-\sigma}U(\tau,t_{0})P_{s}(t_{0})u\right\Vert _{L_{t}^{2}\left((0,T),\, L_{x}^{2}\right)\;\;\;\;\;\;\;\;\;\;\;\;\;\;\;\;\;}^{2}\\
\lesssim h(B)AC^{2}(T)\left\Vert u\right\Vert _{L^{2}}^{2}\;\;\;\;\;\;\;\;\;\;\;\;\;\;\;\;\;\;\;\;\;\;\;\;\;\;\;\;\;\;\;\;\;\;
\end{eqnarray*}
where as before, 
\[
h(B)\lesssim B^{-1}.
\]
Therefore, when we pick $B$ large enough, we have satisfied all the conditions
for the bootstrap argument. 

Finally, we analyze $J$:
\[
J=\int_{t-A}^{t}\int_{s-B}^{s}e^{i\frac{1}{2}\Delta(t-s)}V_{1}P_{c}\left(H_{1}\right)e^{-iH_{1}(s-\tau)}V_{2}(\cdot-\tau\vec{e_{1}})U(\tau,t_{0})P_{s}(t_{0})\,d\tau ds.
\]
We decompose the integral into low and high frequency parts:
\[
J_{L}=\int_{t-A}^{t}\int_{s-B}^{s}e^{i\frac{1}{2}\Delta(t-s)}V_{1}F\left(\left|\vec{p}\right|\leq M\right) e^{-iH_{1}(s-\tau)}P_{c}\left(H_{1}\right)V_{2}(\cdot-\tau\vec{e_{1}})U(\tau,t_{0})P_{s}(t_{0})\,d\tau ds,
\]
\[
J_{H}=\int_{t-A}^{t}\int_{s-B}^{s}e^{i\frac{1}{2}\Delta(t-s)}V_{1}F\left(\left|\vec{p}\right|\geq M\right)e^{-iH_{1}(s-\tau)}P_{c}\left(H_{1}\right)V_{2}(\cdot-\tau\vec{e_{1}})U(\tau,t_{0})P_{s}(t_{0})\,d\tau ds,
\]
where $F\left(\left|\vec{p}\right|\leq M\right)$ and $F\left(\left|\vec{p}\right|\geq M\right)$
denote smooth projections onto frequencies $\left|\vec{p}\right|\leq M$
and $\left|\vec{p}\right|\geq M$ respectively.

To analyze the low frequency part, we observe that for arbitrary $\epsilon>0$,
\[
\int_{s-B}^{s}\left\langle -\tau\vec{e}_{1}\right\rangle ^{-\frac{3}{2}}d\tau\leq\epsilon
\]
provided $s$ is large enough. 

Set $V_{2}^{\sigma}(x)=V_{2}(x)\left\langle x\right\rangle ^{\sigma}$, then we look at the following quantity,
\[
\left\Vert V_{1}F\left(\left|\vec{p}\right|\leq M\right)e^{-i\frac{1}{2}(s-\tau)\Delta}V_{2}^{\sigma}\left(\cdot-\tau\vec{e}_{1}\right)u\right\Vert _{L^{2}}=\left\Vert \int_{\mathbb{R}^{3}}K(x,\eta)\hat{u}(\eta)\,d\eta\right\Vert _{L^{2}}
\]
where 
\[
K(x,\eta)=V_{1}(x)\int_{\mathbb{R}^{3}}e^{-i\frac{1}{2}(s-\tau)\xi^{2}+i\xi(x+\tau\vec{e}_{1})}\chi\left(\frac{\xi}{M}\right)\widehat{V_{2}^{\sigma}}(\xi-\eta) e^{-i\eta\tau\vec{e}_{1}}\,d\xi.
\]
Observe that 
\[
\left|K(x,\eta)\right|\leq C_{M}\left\langle x\right\rangle ^{-N}\left\langle \tau\vec{e}_{1}\right\rangle ^{-N}\left\langle \eta\right\rangle ^{-N}.
\]
This decay result follows from the following two facts: 

Integration by parts with
\[
e^{iI_{2}\xi\tau\vec{e}_{1}}=\left(\frac{\tau\vec{e}_{1}\nabla_{\xi}}{\left|\tau\vec{e}_{1}\right|^{2}}\right)^{N}e^{I_{2}\xi\tau\vec{e}_{1}};
\]
and the decay estimate:
\[
\left|D^{\beta}\widehat{V_{2}^{\sigma}}(\xi-\eta)\right|\lesssim\left\langle \eta\right\rangle ^{-N},\,\,\left|\xi\right|\lesssim M.
\]
So we can conclude that for any $N>0$,
\[
\left\Vert V_{1}F\left(\left|\vec{p}\right|\leq M\right)e^{-i\frac{1}{2}(s-\tau)\Delta}V_{2}^{\sigma}\left(\cdot-\tau\vec{e}_{1}\right)\right\Vert _{2\rightarrow2}\leq C_{N,M}\left\langle \tau\vec{e}_{1}\right\rangle ^{-N}.
\]
By some similar calculations in \cite{Graf},  we conclude 
\[
\left\Vert V_{1}F\left(\left|\vec{p}\right|\leq M\right)e^{-i(s-\tau)H_{1}}P_{c}\left(H_{1}\right)V_{2}^{\sigma}\left(\cdot-\tau\vec{e}_{1}\right)\right\Vert _{2\rightarrow2}\leq C_{N,M}\left\langle \tau\vec{e}_{1}\right\rangle ^{-N}.
\]
But in our particular situation, one can do easy calculations based
on Duhamel formula, 
\[
e^{-i(s-\tau)H_{1}}=e^{-i(s-\tau)H_{0}}-i\int_{\tau}^{s}e^{-i(s-\tau)H_{0}}V_{1}e^{-i(r-\tau)H_{1}}\,dr
\]
\begin{eqnarray*}
	F\left(\left|\vec{p}\right|\leq M\right)e^{-i(s-\tau)H_{1}}P_{c}\left(H_{1}\right)V_{2}^{\sigma}\left(\cdot-\tau\vec{e}_{1}\right)\\
	=F\left(\left|\vec{p}\right|\leq M\right)e^{-i(s-\tau)H_{0}}V_{2}^{\sigma}\left(\cdot-\tau\vec{e}_{1}\right)\\
	-i\int_{\tau}^{s}F\left(\left|\vec{p}\right|\leq M\right)e^{-i(s-\tau)H_{0}}V_{1}e^{-i(r-\tau)H_{1}}V_{2}^{\sigma}\left(\cdot-\tau\vec{e}_{1}\right)\,dr.
\end{eqnarray*}
\begin{eqnarray*}
	\int_{\tau}^{s}F\left(\left|\vec{p}\right|\leq M\right)e^{-i(s-\tau)H_{0}}V_{1}e^{-i(r-\tau)H_{1}}V_{2}^{\sigma}\left(\cdot-\tau\vec{e}_{1}\right)\,dr\\
	=\int_{\tau}^{s}e^{-i(s-\tau)H_{0}}V_{1}F\left(\left|\vec{p}\right|\leq M\right)e^{-i(r-\tau)H_{1}}V_{2}^{\sigma}\left(\cdot-\tau\vec{e}_{1}\right)\,dr\\
	+\int_{\tau}^{s}e^{-i(s-\tau)H_{0}}\left[V_{1},F\left(\left|\vec{p}\right|\leq M\right)\right]e^{-i(r-\tau)H_{1}}V_{2}^{\sigma}\left(\cdot-\tau\vec{e}_{1}\right)\,dr.
\end{eqnarray*}
\begin{eqnarray*}
	\left\Vert \int_{\tau}^{s}F\left(\left|\vec{p}\right|\leq M\right)e^{-i(s-\tau)H_{0}}V_{1}e^{-i(r-\tau)H_{1}}V_{2}^{\sigma}\left(\cdot-\tau\vec{e}_{1}\right)\,dr\right\Vert _{L^{2}}\\
	\leq\int_{\tau}^{s}\left\Vert V_{1}F\left(\left|\vec{p}\right|\leq M\right)e^{-i(r-\tau)H_{1}}V_{2}^{\sigma}\left(\cdot-\tau\vec{e}_{1}\right)\right\Vert _{L^{2}}\,dr\\
	\left\Vert \int_{\tau}^{s}e^{-i(s-\tau)H_{0}}\left[V_{1},F\left(\left|\vec{p}\right|\leq M\right)\right]e^{-i(r-\tau)H_{1}}V_{2}^{\sigma}\left(\cdot-\tau\vec{e}_{1}\right)\,dr\right\Vert _{L^{2}}.
\end{eqnarray*}
Notice from construction, $0\leq s-\tau\leq B$,
\[
\left\Vert \int_{\tau}^{s}e^{-i(s-\tau)H_{0}}\left[V_{1},F\left(\left|\vec{p}\right|\leq M\right)\right]e^{-i(r-\tau)H_{1}}V_{2}^{\sigma}\left(\cdot-\tau\vec{e}_{1}\right)\,dr\right\Vert _{L^{2}}\lesssim\frac{B}{M}\left\Vert V_{2}^{\sigma}\right\Vert _{L^{2}}
\]
By Gronwall's inequality, with the fact $0\leq s-\tau\leq B$, one
has 
\begin{eqnarray}
\int_{s-B}^{s}\left\Vert V_{1}F\left(\left|\vec{p}\right|\leq M\right)e^{-i(s-\tau)H_{1}}P_{c}\left(H_{1}\right)V_{2}^{\sigma}\left(\cdot-\tau\vec{e}_{1}\right)\right\Vert _{2\rightarrow2}\nonumber \\
\lesssim e^{B}\int_{s-B}^{s}\left(C_{N,M}\left\langle \tau\vec{e}_{1}\right\rangle ^{-N}+\frac{B}{M}\left\Vert V_{2}^{\sigma}\right\Vert _{L^{2}}\right)\,d\tau\\
\lesssim\epsilon+\frac{B^{2}e^{B}}{M}\nonumber \\
\lesssim\epsilon\nonumber 
\end{eqnarray}
provided $M$ is large enough. 

Therefore, for $J_{L}$,
\begin{eqnarray*}
\left\Vert \langle x-x_{0}\rangle^{-\sigma}J_{L}\langle x-x_{1}\rangle^{-\sigma}\right\Vert _{2\rightarrow2}\;\;\;\;\;\;\;\;\;\;\;\;\;\;\;\;\;\;\;\;\;\;\;\;\;\;\;\;\;\;\;\;\;\;\;\;\;\;\;\;\;\;\;\;\;\;\;\;\;\;\;\;\;\;\;\;\;\;\;\;\;\;\;\;\;\;\;\;\;\;\;\;\;\;\;\;\;\;\;\;\\
\leq CC(T)\int_{t-A}^{t}\int_{s-B}^{s}\left\langle t-s\right\rangle ^{-\frac{3}{2}}\left\langle s-\tau\right\rangle ^{-\frac{3}{2}}\left\langle \tau\vec{e}_{1}\right\rangle ^{-\frac{3}{2}}dsd\tau\;\;\;\;\;\;\;\;\;\;\;\;\;\;\;\;\;\;\;\;\\
\leq CC(T)\left\langle t-t_{0}\right\rangle ^{-\frac{3}{2}}\int_{t-A}^{t}\int_{s-B}^{s}\left\langle t-s\right\rangle ^{-\frac{3}{2}}\left\langle \tau\vec{e}_{1}\right\rangle^{-\frac{3}{2}} dsd\tau\;\;\;\;\;\;\;\;\;\;\;\;\;\;\;\;\;\;\\
\leq\epsilon CC(T)\left\langle t-t_{0}\right\rangle ^{-\frac{3}{2}}\;\;\;\;\;\;\;\;\;\;\;\;\;\;\;\;\;\;\; &  & .
\end{eqnarray*}
So when $A,\, B$ is large, we conclude that the coefficient satisfies the
bootstrap conditions.

For the second part,

\begin{eqnarray*}
\left\Vert \langle x-x(t)\rangle^{-\sigma}J_{L}u_{0}\right\Vert _{L^{2}\left(\left(0,T\right),L_{x}^{2}\right)}\;\;\;\;\;\;\;\;\;\;\;\;\;\;\;\;\;\;\;\;\;\;\;\;\;\;\;\;\;\;\;\;\;\;\;\;\;\;\;\;\;\;\;\;\;\;\;\;\;\;\;\;\;\;\;\;\;\;\;\;\;\\
\leq C\left\Vert \int_{t-A}^{t}ds\int_{s-B}^{B}\left\langle t-s\right\rangle ^{-\frac{3}{2}}\left\langle -\tau\vec{e}_{1}\right\rangle ^{-\frac{3}{2}}\left\Vert \langle x-\tau\vec{e_{1}}\rangle^{-\sigma}U(\tau,t_{0})P_{s}(t_{0})u_{0}\right\Vert \right\Vert \\
\leq CC(T)\sqrt{A}\left\Vert \left\langle -\tau\vec{e}_{1}\right\rangle ^{-\frac{3}{2}}\right\Vert _{L^{2}(s-B,s)}\left\Vert U(\tau,t_{0})P_{s}(t_{0})u_{0}\right\Vert _{L^{\infty}\left(\left(0,T\right),L_{x}^{2}\right)}\\
\leq CC(T)\sqrt{A\epsilon}\left\Vert u_{0}\right\Vert _{L^{2}}.
\end{eqnarray*}
Again, we know when $\epsilon$ is small, we recapture the bootstrap
conditions.

Finally, we need to check $J_{c,H}$. We will use the following version
of the Kato smoothing estimate, or we can apply a variant of Kato's smoothing
estimate from \cite{RSS}.
\begin{lem}
[\cite{LP}]\label{lem:kato} For $\sigma>\frac{1}{2}$, we have

\[
\intop_{\mathbb{R}}\left\Vert \left\langle x\right\rangle ^{-\sigma}\left\langle \vec{p}\right\rangle ^{-\frac{1}{2}}\nabla e^{-i\frac{1}{2}(s-\tau)\Delta}\right\Vert _{2\rightarrow2}^{2}d\tau\leq C_{\tau},
\]
we also have 
\[
\intop_{\mathbb{R}}\left\Vert \left\langle x\right\rangle ^{-\sigma}\left\langle \vec{p}\right\rangle ^{-\frac{1}{2}}\nabla e^{-i(s-\tau)H_{1}}P_{c}\left(H_{1}\right)\right\Vert _{2\rightarrow2}^{2}d\tau\leq C_{\tau,V_{1}}.
\]

\end{lem}
We will use Lemma \ref{lem:kato}, but for the sake of completeness,
we formulate the result from \cite{RSS}.
\begin{lem*}
[\cite{RSS}] Let $H=-\frac{1}{2}\Delta+V$, $\left\Vert V\right\Vert _{\infty}<\infty$,
set $\psi(t)=e^{-itH}\psi_{0}$, then for all $T>1$ and $\alpha>0$,
we have 
\[
\sup_{x_{0}\in\mathbb{R}^{n}}\int_{0}^{T}\int_{\mathbb{R}^{n}}\frac{\left|\nabla\left\langle \nabla\right\rangle ^{-\frac{1}{2}}\psi(x,t)\right|^{2}}{\left(1+\left|x-x_{0}\right|^{\alpha}\right)^{\frac{1}{\alpha}+1}}dxdt\leq C_{\alpha,n}T\left(1+\left\Vert V\right\Vert _{\infty}\right)\left\Vert \psi_{0}\right\Vert _{L^{2}.}
\]

\end{lem*}
Consider 
\begin{eqnarray*}
\int_{s-B}^{s}\left\Vert V_{1}F\left(\left|\vec{p}\right|\geq M\right)e^{-iH_{1}(s-\tau)}P_{c}\left(H_{1}\right)V_{2}(\cdot-\tau\vec{e_{1}})\right\Vert _{2\rightarrow2}d\tau\;\;\;\;\;\;\;\;\;\;\;\;\;\;\;\;\;\;\;\\
\leq\int_{s-B}^{s}\left\Vert \left[V_{1},F\left(\left|\vec{p}\right|\geq M\right)\right]e^{-iH_{1}(s-\tau)}P_{c}\left(H_{1}\right)V_{2}(\cdot-\tau\vec{e_{1}})\right\Vert _{2\rightarrow2}d\tau\\
+B^{\frac{1}{2}}M^{-\frac{1}{2}}\left(\int_{s-B}^{s}\left\Vert \left\langle \vec{p}\right\rangle ^{-\frac{1}{2}}\nabla V_{1}e^{-iH_{1}(s-\tau)}P_{c}\left(H_{1}\right)V_{2}(\cdot-\tau\vec{e_{1}})\right\Vert _{2\rightarrow2}d\tau\right)^{\frac{1}{2}}.
\end{eqnarray*}
By Young's inequality \cite{RSS,Cai}, 
\[
\left\Vert \left[V_{1},F\left(\left|\vec{p}\right|\geq M\right)\right]\right\Vert _{2\rightarrow2}\lesssim\frac{1}{M}.
\]
Also note that 
\[
\left\langle \vec{p}\right\rangle ^{-\frac{1}{2}}\nabla V_{1}=V_{1}\left\langle \vec{p}\right\rangle ^{-\frac{1}{2}}\nabla+\left[\left\langle \vec{p}\right\rangle ^{-\frac{1}{2}}\nabla,V_{1}\right],
\]
and 
\[
\left\Vert \left[\left\langle \vec{p}\right\rangle ^{-\frac{1}{2}}\nabla,V_{1}\right]\right\Vert _{2\rightarrow2}\lesssim1.
\]
So for the first estimate, with bootstrap assumption \eqref{eq:boot1},
we get 
\begin{eqnarray*}
\left\Vert \langle x-x_{0}\rangle^{-\sigma}J_{H}\langle x-x_{1}\rangle^{-\sigma}\right\Vert _{2\rightarrow2}\;\;\;\;\;\;\;\;\;\;\;\;\;\;\;\;\;\;\;\;\;\;\;\;\;\;\;\;\;\;\;\;\;\;\;\;\;\;\;\;\;\;\;\;\;\;\;\;\;\;\;\\
\lesssim C(T)\left\langle t-t_{0}\right\rangle ^{-\frac{3}{2}}\int_{t-A}^{t}\left\langle t-s\right\rangle ^{-\frac{3}{2}}\;\;\;\;\;\;\;\;\;\;\;\;\;\;\;\;\;\;\;\;\;\;\;\;\;\;\;\;\;\;\;\;\;\;\;\;\;\;\;\;\;\;\;\;\;\;\;\;\;\;\;\\
\times\int_{s-B}^{s}\left\Vert V_{1}F\left(\left|\vec{p}\right|\geq M\right)e^{-iH_{1}(s-\tau)}P_{c}\left(H_{1}\right)V_{2}(\cdot-\tau\vec{e_{1}})\right\Vert _{2\rightarrow2}d\tau ds\\
\lesssim C(T)\left\langle t-t_{0}\right\rangle ^{-\frac{3}{2}}\int_{t-A}^{t}\left\langle t-s\right\rangle ^{-\frac{3}{2}}\left(M^{-1}B+M^{-\frac{1}{2}}\sqrt{B}(1+\sqrt{B})\right)\\
\lesssim C(T)ABM^{-\frac{1}{2}}\left\langle t-t_{0}\right\rangle ^{-\frac{3}{2}}.
\end{eqnarray*}
For the other estimate,
\begin{eqnarray*}
\left\Vert \left\langle x-x(t)\right\rangle ^{-\sigma}J_{H}u_{0}\right\Vert _{L^{2}\left(\left(0,T\right),L_{x}^{2}\right)}\;\;\;\;\;\;\;\;\;\;\;\;\;\;\;\;\;\;\;\;\;\;\;\;\;\;\;\;\;\;\;\;\;\;\;\;\;\;\;\;\;\;\;\;\;\;\;\;\;\;\;\;\;\;\;\;\;\;\;\;\;\;\;\;\;\;\;\;\;\;\;\;\;\;\;\;\;\;\;\;\;\;\;\;\;\;\;\;\;\;\;\;\;\;\;\;\;\;\;\;\;\;\;\;\;\;\;\;\;\;\;\;\;\;\;\\
\lesssim\left\Vert \int_{t-A}^{t}\left\langle t-s\right\rangle ^{-\frac{3}{2}}\int_{s-B}^{s}\left\Vert V_{1}F\left(\left|\vec{p}\right|\geq M\right)e^{-iH_{1}(s-\tau)}P_{c}\left(H_{1}\right)V_{2}(\cdot-\tau\vec{e_{1}})\right\Vert _{2\rightarrow2}\left\Vert \left\langle x-\tau\vec{e_{1}}\right\rangle ^{-\sigma}U(\tau,t_{0})P_{s}(t_{0}u_{0}\right\Vert _{L_{x}^{2}}\right\Vert \;\;\;\;\;\;\;\;\;\;\;\;\;\;\\
\lesssim\left\Vert \int_{t-A}^{t}\left\langle t-s\right\rangle ^{-\frac{3}{2}}\int_{s-B}^{s}\left(M^{-1}+M^{-\frac{1}{2}}(B+\sqrt{B})\right)\left\Vert \langle x-\tau\vec{e_{1}}\rangle^{-\sigma}U(\tau,t_{0})P_{s}(t_{0})u_{0}\right\Vert \right\Vert _{L^{2}(0,T)}\;\;\;\;\;\;\;\;\;\;\;\;\;\;\;\;\;\;\;\;\;\;\;\;\;\;\\
\lesssim C(T)BM^{-\frac{1}{2}}\left\Vert u_{0}\right\Vert _{L^{2}}.\;\;\;\;\;\;\;\;\;\;\;\;\;\;\;\;\;\;\;\;\;\;\;\;\;\;\;\;\;\;\;\;\;\;\;\; &  & .
\end{eqnarray*}
So we can pick $M$ large, then the coefficient satisfies the bootstrap condition again.

To sum up, when we pick $A$, $B$ and $M$ large enough independent
of $T$, if we have for $t\in[t_{0},T]$
\begin{eqnarray*}
\left\Vert \langle x-x_{0}\rangle^{-\sigma}U(t,t_{0})P_{s}\langle x-x_{1}\rangle^{-\sigma}\right\Vert _{2\rightarrow2}\leq C(T)\frac{1}{\langle t-t_{0}\rangle^{\frac{3}{2}}},
\end{eqnarray*}
we can improve it to 
\begin{eqnarray*}
\left\Vert \langle x-x_{0}\rangle^{-\sigma}U(t,t_{0})P_{s}\langle x-x_{1}\rangle^{-\sigma}\right\Vert _{2\rightarrow2}\leq\frac{1}{2}C(T)\frac{1}{\langle t-t_{0}\rangle^{\frac{3}{2}}}+C\frac{1}{\langle t-t_{0}\rangle^{\frac{3}{2}}}.
\end{eqnarray*}
Therefore, we can make for $t\in[t_{0},T]$, 
\[
\left\Vert \langle x-x_{0}\rangle^{-\sigma}U(t,t_{0})P_{s}\langle x-x_{1}\rangle^{-\sigma}\right\Vert _{2\rightarrow2}\leq C,
\]
for some constant independent of $T$. So we conclude 
\[
\left\Vert \langle x-x_{0}\rangle^{-\sigma}U(t,t_{0})P_{s}\langle x-x_{1}\rangle^{-\sigma}\right\Vert _{2\rightarrow2}\leq C
\]
holds for arbitrary $t$ which shows Lemma \ref{lem:We1}. 

For the second part we proceed analogously. Indeed, if we suppose 
\begin{eqnarray*}
\int_{0}^{T}\left\Vert \left\langle x-x(t)\right\rangle ^{-\sigma}U(t,t_{0})P_{s}u\right\Vert _{L_{x}^{2}}^{2} dt& \leq & C^{2}(T)\left\Vert u\right\Vert _{L^{2},}^{2}
\end{eqnarray*}
then we can improve the estimate to 
\[
\int_{0}^{T}\left\Vert \left\langle x-x(t)\right\rangle ^{-\sigma}U(t,t_{0})P_{s}u\right\Vert _{L_{x}^{2}}^{2}dt\leq C\left\Vert u\right\Vert _{L^{2}}^{2}+\frac{1}{2}C^{2}(T)\left\Vert u\right\Vert _{L^{2}}^{2}.
\]
So we can obtain a bound for 
\[
\int_{0}^{T}\left\Vert \left\langle x-x(t)\right\rangle ^{-\sigma}U(t,t_{0})P_{s}u\right\Vert _{L_{x}^{2}}^{2}dt\leq C\left\Vert u\right\Vert _{L^{2}}^{2}
\]
which is independent of $T$. Therefore, we can send $T$ to $\infty$ above. Finally, we obtain
\[
\int_{0}^{\infty}\left\Vert \left\langle x-x(t)\right\rangle ^{-\sigma}U(t,t_{0})P_{s}u\right\Vert _{L_{x}^{2}}^{2}dt\leq C\left\Vert u\right\Vert _{L^{2}}^{2}
\]
which establishes Lemma \ref{lem:We2}.

\begin{rem}
With Theorem \ref{thm:dispersive}, one can show Lemma \ref{lem:We1}
easily as the free case. Set $s=t-t_{0}$, first, if $|s|\leq1$,
clearly by $\left\Vert U(t,t_{0})\right\Vert _{2\rightarrow2}\leq1$
and the integrability condition in $\mathbb{R}^{3}$, i.e. $\sigma>\frac{3}{2}$,
we can get the desired result. If $|s|\geq1$, we apply the dispersive
estimate for the free motion, by Young's inequality we get 
\[
\left\Vert \langle x-x_{0}\rangle^{-\sigma}U(t,t_{0})P_{s}(t_{0})\langle x-x_{1}\rangle^{-\sigma}\right\Vert _{2\rightarrow2}\lesssim\left\Vert \langle x\rangle^{-\sigma}\right\Vert _{L^{2}}^{2}\left\Vert U(t,t_{0})P_{s}(t_{0})\right\Vert _{1\rightarrow\infty},
\]
and from Theorem \ref{thm:dispersive},
\[
\left\Vert U(t,t_{0})P_{s}(t_{0})\right\Vert _{1\rightarrow\infty}\lesssim|t-t_{0}|^{-\frac{3}{2}}.
\]
But we proved Lemma \ref{lem:We1} together with Lemma \ref{lem:We2},
since the dispersive estimate might not be available in other contexts.
	
\end{rem}

\section{Boundedness of The Energy}\label{sec:energy}

In this section, we use Strichartz estimates to show that the energy of
the whole evolution of the charge transfer model is bounded independently
of time. The asymptotic completeness of the Hamiltonian shown
in \cite{RSS} will be used. We will still consider the model with two potentials
as in the previous section.

\begin{proof}[Proof of Theorem \ref{thm:energy}]
From Theorem \ref{thm:ac}, we can write the evolution
as: for some $\phi_{0}\in L^{2}\left(\mathbb{R}^{3}\right)$, 
\begin{equation}
U(t,0)\psi_{0}=\sum_{r=1}^{m}A_{r}e^{-i\lambda_{r}t}u_{r}+\sum_{s=1}^{\ell}B_{s}e^{-i\mu_{s}t}\mathfrak{g}_{-\overrightarrow{e_{1}}}(t)w_{s}+e^{-it\frac{\Delta}{2}}\phi_{0}+R(t)\label{eq:31}
\end{equation}
where $\mathfrak{g}$ is the Galilei transformation. It is trivial
to see the part associated with bound states and moving bound states,
\begin{equation}
\sum_{r=1}^{m}A_{r}e^{-i\lambda_{r}t}w_{r}+\sum_{s=1}^{\ell}B_{s}e^{-i\mu_{s}t}\mathfrak{g}_{-\overrightarrow{e_{1}}}(t)u_{s}\label{eq:33}
\end{equation}
has bounded energy. 
Indeed, to be more precise, we have 
\[
\left\Vert \sum_{r=1}^{m}A_{r}e^{-i\lambda_{r}t}w_{r}+\sum_{s=1}^{\ell}B_{s}e^{-i\mu_{s}t}\mathfrak{g}_{-\overrightarrow{e_{1}}}(t)u_{s}\right\Vert _{H^{1}}\lesssim\sum_{r=1}^{m}\left\Vert w_{r}\right\Vert _{H^{1}}\left\Vert \psi_{0}\right\Vert_{L^{2}} +\sum_{s=1}^{\ell}\left\Vert u_{s}\right\Vert _{H^{1}}\left\Vert \psi_{0}\right\Vert_{L^{2}}.
\]
So it suffices to consider 
\begin{equation}
\psi(t):=U(t,0)\psi_{0}=e^{-it\frac{\Delta}{2}}\phi_{0}+R(t)\label{eq:34}
\end{equation}
where 
\[
\left\Vert R(t)\right\Vert _{L^{2}}\rightarrow0,\,\, t\rightarrow\pm\infty.
\]
In other words, we might assume 
\[
P_s(t)\psi(t)=\psi(t).
\]
Rewrite the equation,
\begin{equation}
i\psi_{t}+\frac{\Delta}{2}\psi=V_{1}\psi+V_{2}(\cdot-t\vec{e_{1}})\psi.\label{eq:35}
\end{equation}
We can differentiate the equation \eqref{eq:35} and set 
\[
v=\partial_{x_{1}}\psi=:\partial_{1}\psi,
\]
then $v$ satisfies 
\begin{equation}
iv_{t}+\Delta v-V_{1}v-V_{2}(\cdot-t\vec{e_{1}})v=\partial_{1}V_{1}\psi+\partial_{1}V_{2}(\cdot-t\vec{e_{1}})\psi.\label{eq:36}
\end{equation}
Again, it suffices to consider $\psi$ is in the scattering space. Since other components are easily to be bounded. To see this, we look at 
\[
\left\langle v,w_{r}\right\rangle _{L^{2}}=-\left\langle e^{-it\frac{\Delta}{2}}\phi_{0}+R(t),\partial_{x_{1}}w_{r}\right\rangle _{L^{2}}.
\]
By the asymptotic completeness result, we know 
\[
\left\Vert R(t)\right\Vert _{L^{2}}\rightarrow0,\,\, t\rightarrow\pm\infty.
\]
In particular, we know 
\[
\left\Vert R(t)\right\Vert _{L^{2}}\lesssim C,
\]
so 
\[
\left|\left\langle R(t),\partial_{x_{1}}w_{r}\right\rangle _{L^{2}}\right|\lesssim\left\Vert R(0)\right\Vert _{L^{2}}\lesssim\left\Vert \psi_{0}\right\Vert _{L^{2}}
\]
since from Agmon's estimate, $\partial_{x_{1}}w_{r}$ is still exponentially
decaying. 

Notice that 

\[
\left|\left\langle e^{-it\frac{\Delta}{2}}\phi_{0},\partial_{x_{1}}w_{r}\right\rangle _{L^{2}}\right|\rightarrow0,\,t\rightarrow\infty,
\]
since we can approximate $\phi_{0}$ by $\phi_{n}\in L^{2}\cap L^{1}$
in $L^{2}$ and then by the dispersive estimate for the free equation 
\[
\left\Vert e^{-it\frac{\Delta}{2}}\phi_{n}\right\Vert _{L^{\infty}}\lesssim\frac{1}{\left|t\right|^{\frac{3}{2}}}\left\Vert \phi_{n}\right\Vert _{L^{1}}.
\]
A similar discussion holds for $u_{s}$, we can conclude that 
\[
\left\Vert P_{b}\left(H_{1}\right)v(t)\right\Vert _{L^{2}}+\left\Vert P_{b}\left(H_{2},t\right)v(t)\right\Vert _{L^{2}}\rightarrow0,\,\,\, t\rightarrow\pm\infty,
\]
and 
\[
\left\Vert P_{b}\left(H_{1}\right)v(t)\right\Vert _{L^{2}}+\left\Vert P_{b}\left(H_{2},t\right)v(t)\right\Vert _{L^{2}}\lesssim\left\Vert \psi_{0}\right\Vert _{L^{2}.}
\]
By the above argument, we can actually conclude that $v$ is asymptotically orthogonal
to the bound states of $H_{1}$ and moving bound states associated
to $H_{2}(t)$. We can in fact obtain an explicit rate of decay for the term
\[
\left\Vert P_{b}\left(H_{1}\right)v(t)\right\Vert _{L^{2}}+\left\Vert P_{b}\left(H_{2},t\right)v(t)\right\Vert _{L^{2}}
\]
goes to $0$, but it is enough for our purposes  to know that it is just bounded by $\left\Vert \psi_{0}\right\Vert _{H^{1}}$.

Then by Proposition \ref{prop:convrate}, 
\[
\left\Vert (1-P_s(t)v(t)\right\Vert _{L^{2}}\lesssim \left\Vert \psi_{0}\right\Vert _{H^{1}}.
\]
Therefore, it is sufficient to estimate 
\[
\left\Vert P_s(t)v(t)\right\Vert _{L^{2}}
\]
and hence, without loss of generality, we assume 
\[
P_s(t)v(t).
\]
We do a similar argument as the proof for Strichartz
estimates, Theorem \ref{thm:Strich}.

Setting $F(x,t)=\partial_{1}V_{1}\psi+\partial_{1}V_{2}(\cdot-t\vec{e_{1}})\psi$,
we  can write \eqref{eq:36} in the form 
\[
iv_{t}+\frac{\Delta}{2}v=V_{1}v+V_{2}(\cdot-t\vec{e_{1}})v+F(x,t).
\]
By the endpoint Strichartz for the free Sch\"odinger equation,
we obtain
\[
\left\Vert v\right\Vert _{L_{t}^{p}L_{x}^{q}}\leq C\left\Vert V_{1}v+V_{2}(\cdot-t\vec{e_{1}})v+F\right\Vert _{L_{t}^{2}L_{x}^{\frac{6}{5}}}
\]
\begin{eqnarray*}
\left\Vert V_{1}v+V_{2}(\cdot-t\vec{e_{1}})v+F\right\Vert _{L_{t}^{2}L_{x}^{\frac{6}{5}}} & \leq & \left\Vert V_{1}v\right\Vert _{L_{t}^{2}L_{x}^{\frac{6}{5}}}+\left\Vert V_{2}(\cdot-t\vec{e_{1}})v\right\Vert _{L_{t}^{2}L_{x}^{\frac{6}{5}}}+\left\Vert F\right\Vert _{L_{t}^{2}L_{x}^{\frac{6}{5}}},\\
\left\Vert V_{1}v\right\Vert _{L_{t}^{2}L_{x}^{\frac{6}{5}}} & \leq & C_{V}\left\Vert \left\langle x\right\rangle ^{-m}v\right\Vert _{L_{t}^{2}L_{x}^{2}},\\
\left\Vert V_{2}(\cdot-t\vec{e_{1}})v\right\Vert _{L_{t}^{2}L_{x}^{\frac{6}{5}}} & \leq & C_{V}\left\Vert \left\langle x-t\vec{e_{1}}\right\rangle ^{-m}v\right\Vert _{L_{t}^{2}L_{x}^{2}}.
\end{eqnarray*}
So it suffices to estimate 
\[
\left\Vert \left\langle x-x(t)\right\rangle ^{-m}v\right\Vert _{L_{t}^{2}L_{x}^{2}}
\]
for $x(t)$ a smooth curve in $\mathbb{R}^{3}$.

By Duhamel's formula, 
\[
v(t)=e^{i\frac{1}{2}\Delta t}v_{0}+i\int_{0}^{t}e^{i\frac{1}{2}\Delta(t-s)}\left(V_{1}+V_{2}(\cdot-s\vec{e_{1}})\right)v(s)\,ds+i\int_{0}^{t}e^{i\frac{1}{2}\Delta(t-s)}F(x,s)\,ds.
\]
We write 
\[
v(t)=U_{1}+iU_{2}+iU_{3}.
\]
Certainly, it is easy to bound $U_{1}$ as Lemmas \ref{lem:FWe1} and
\ref{lem:Fwe2}. 

Next we bound $U_{3}$. We again apply H\"older's inequality and the endpoint
Strichartz estimate, 
\begin{eqnarray*}
\left\Vert \left\langle x-x(t)\right\rangle ^{-m}\int_{t_{0}}^{t}e^{i\frac{1}{2}\Delta(t-s)}F(x,s)\,ds\right\Vert _{L_{t}^{2}L_{x}^{2}} & \leq & \left\Vert \int_{t_{0}}^{t}e^{i\frac{1}{2}\Delta(t-s)}F(x,s)\,ds\right\Vert _{L_{t}^{2}L_{x}^{6}}\\
 & \leq & \left\Vert F\right\Vert _{L_{t}^{2}L_{x}^{\frac{6}{5}}}.
\end{eqnarray*}
It remains to bound 
\[
\left\Vert \int_{0}^{t}e^{i\frac{1}{2}\Delta(t-s)}\left(V_{1}+V_{2}(\cdot-s\vec{e_{1}})\right)v(s)\,ds\right\Vert _{L_{t}^{2}L_{x}^{2}}.
\]
Rewrite 
\[
v(s)=U(s,0)v_{0}+i\int_{0}^{s}U(s,\tau)F(x,\tau)\,d\tau
\]
By our assumption:
\[
v(s)=P_s(s)U(s,0)v_{0}+i\int_{0}^{s}P_s(s)U(s,\tau)F(x,\tau)\,d\tau.
\]
By Lemma \ref{lem:We1}, Lemma \ref{lem:We2}, we have
\begin{eqnarray*}
\left\Vert \left\langle x-x(t)\right\rangle ^{-m}\int_{0}^{t}e^{i\frac{1}{2}\Delta(t-s)}\left(V_{1}+V_{2}(\cdot-s\vec{e_{1}})\right)P_s(s)U(s,0)v_{0}\,ds\right\Vert _{L_{t}^{2}L_{x}^{2}}\;\;\;\;\;\;\;\;\;\;\;\;\;\;\;\;\;\;\;\;\;\;\\
\lesssim\left\Vert \left\langle x\right\rangle ^{-\beta}P_s(s)U(t,0)v_{0}\right\Vert _{L_{t}^{2}L_{x}^{2}}+\left\Vert \left\langle x-t\vec{e_{1}}\right\rangle ^{-\beta}P_s(s)U(t,0)v_{0}\right\Vert _{L_{t}^{2}L_{x}^{2}}\\
\lesssim\left\Vert v_{0}\right\Vert _{L^{2}}.\;\;\;\;\;\;\;\;\;\;\;\;\;\;\;\;\;\;\;\;\;\;\;\;\;\;\;\;\;\;\;\;\;\;\;
\end{eqnarray*}
Also, we can get 

\begin{eqnarray*}
\left\Vert \left\langle x-x(t)\right\rangle ^{-m}\int_{0}^{t}e^{i\frac{1}{2}\Delta(t-s)}\left(V_{1}+V_{2}(\cdot-s\vec{e_{1}})\right)\int_{0}^{s}P_s(s)U(s,\tau)F(x,\tau)\,d\tau ds\right\Vert _{L_{t}^{2}L_{x}^{2}}\\
\lesssim\left\Vert \int_{0}^{t}\left\langle t-s\right\rangle ^{-\frac{3}{2}}\int_{0}^{s}\left\langle s-\tau\right\rangle ^{-\frac{3}{2}}\left\Vert \left\langle x\right\rangle ^{\alpha}F(\tau)\right\Vert _{L_{x}^{2}}d\tau ds\right\Vert _{L_{t}^{2}}\\
\lesssim\left\Vert \int_{0}^{t}\left\langle t-\tau\right\rangle ^{-\frac{3}{2}}\left\Vert \left\langle x\right\rangle ^{\alpha}F(\tau)\right\Vert _{L_{x}^{2}}d\tau\right\Vert \\
\leq\left\Vert \left\langle x\right\rangle ^{\alpha}F(\tau)\right\Vert _{L_{t}^{2}L_{x}^{2}}. &  & 
\end{eqnarray*}
So we have shown that
\[
\left\Vert v\right\Vert _{L_{t}^{p}L_{x}^{q}}\leq C\left\Vert V_{1}v+V_{2}(\cdot-t\vec{e_{1}})v+F\right\Vert _{L_{t}^{2}L_{x}^{\frac{6}{5}}}\lesssim\left\Vert \psi_{0}\right\Vert _{L^{2}}+\left\Vert \left\langle x\right\rangle ^{\alpha}F\right\Vert _{L_{t}^{2}L_{x}^{2}}+\left\Vert F\right\Vert _{L_{t}^{2}L_{x}^{\frac{6}{5}}}
\]
for any  Schr\"odinger admissible pair $(p,q)$.

Plugging in $F=\partial_{1}V_{1}\psi+\partial_{1}V_{2}(\cdot-t\vec{e_{1}})\psi$,
it is easy to estimate 
\[
\left\Vert \left\langle x\right\rangle ^{\alpha}F\right\Vert _{L_{t}^{2}L_{x}^{2}}+\left\Vert F\right\Vert _{L_{t}^{2}L_{x}^{\frac{6}{5}}}\lesssim\left\Vert \psi_{0}\right\Vert _{L^{2}}.
\]
For the second piece, we use 
\begin{eqnarray*}
\left\Vert \partial_{1}V_{1}\psi+\partial_{1}V_{2}(\cdot-t\vec{e_{1}})\psi\right\Vert _{L_{t}^{2}L_{x}^{\frac{6}{5}}} & \leq & \left\Vert \partial_{1}V_{1}\psi\right\Vert _{L_{t}^{2}L_{x}^{\frac{6}{5}}}+\left\Vert \partial_{1}V_{2}(\cdot-t\vec{e_{1}})\psi\right\Vert _{L_{t}^{2}L_{x}^{\frac{6}{5}}}\\
\left\Vert \partial_{1}V_{1}\psi\right\Vert _{L_{t}^{2}L_{x}^{\frac{6}{5}}} & \leq & C_{V}\left\Vert \left\langle x\right\rangle ^{-m}\psi\right\Vert _{L_{t}^{2}L_{x}^{2}}\lesssim\left\Vert \psi_{0}\right\Vert _{L^{2}}\\
\left\Vert \partial_{1}V_{2}(\cdot-t\vec{e_{1}})\psi\right\Vert _{L_{t}^{2}L_{x}^{\frac{6}{5}}} & \leq & C_{V}\left\Vert \left\langle x-t\vec{e_{1}}\right\rangle ^{-m}\psi\right\Vert _{L_{t}^{2}L_{x}^{2}}\lesssim\left\Vert \psi_{0}\right\Vert _{L^{2}}.
\end{eqnarray*}
For the first piece, by H\"older's inequality, we have
\[
\left\Vert \left\langle x\right\rangle ^{\alpha}\left(\partial_{1}V_{1}\psi+\partial_{1}V_{2}(\cdot-t\vec{e_{1}})\psi\right)\right\Vert _{L_{x}^{2}}\lesssim\left\Vert \psi\right\Vert _{L_{x}^{6}.}
\]
Then applying the endpoint Strichart estimate to $\psi$ by Theorem \ref{thm:Strich},
we get 
\[
\left\Vert \left\langle x\right\rangle ^{\alpha}\left(\partial_{1}V_{1}\psi+\partial_{1}V_{2}(\cdot-t\vec{e_{1}})\psi\right)\right\Vert _{L_{t}^{2}L_{x}^{2}}\lesssim\left\Vert \psi\right\Vert _{L_{t}^{2}L_{x}^{6}}\lesssim\left\Vert \psi_{0}\right\Vert _{L_{x}^{2}}.
\]
So in particular, we infer that 
\begin{equation}
\left\Vert v\right\Vert _{L_{t}^{\infty}L_{x}^{2}}\lesssim\left\Vert v_{0}\right\Vert _{L^{2}}+\left\Vert \psi_{0}\right\Vert _{L^{2}}=\left\Vert \psi_{0}\right\Vert _{H^{1}}.\label{eq:38}
\end{equation}
The same argument applies to all other partial derivatives of $\psi$.
So we can conclude that 
\begin{equation}
\sup_{t\in\mathbb{R}}\left\Vert \psi(t)\right\Vert _{H^{1}}\lesssim\left\Vert \psi_{0}\right\Vert _{H^{1}}.\label{eq:39}
\end{equation}
The theorem is proved
\end{proof}
By a simple inductive argument, we obtain the following corollary:
\begin{cor}
For $\psi_{0}\in H^{k}\left(\mathbb{R}^{3}\right)$ where $k$ is a non-negative integer, then
\begin{equation}
\sup_{t\in\mathbb{R}}\left\Vert U(t,0)\psi_{0}\right\Vert _{H^{k}}\leq C\left\Vert \psi_{0}\right\Vert _{H^{k}}.\label{eq:310}
\end{equation}

\end{cor}
\begin{rem*}
As a concluding remark, we notice that we proved the boundedness of the energy based on Strichartz estimates
and the asymptotic completeness of the Hamiltonian. In \cite{Graf}, Graf proved the asymptotic completeness based on
the boundedness of the energy. So, we can see, modulo some technical assumptions on the spectrum of the Sch\"odinger operator, the
boundedness of the energy is equivalent to the asymptotic completeness of the Hamiltonian. Also note that the asymptotic completeness can be also proved by the dispersive estimate as in \cite{RSS}.  
\end{rem*}

\section{Matrix Charge Transfer Models}\label{sec:matrix}

In this section, we extend our above results to matrix charge transfer
models in $\mathbb{R}^{3}$ similarly as the work in \cite{RSS}.
For the sake of completeness, we start from the basic definitions
following \cite{RSS}.
\begin{defn}
\label{Matrix}By a matrix charge transfer model we mean a system
\[
\frac{1}{i}\partial_{t}\vec{\psi}+\left(\begin{array}{cc}
-\frac{1}{2}\Delta & 0\\
0 & \frac{1}{2}\Delta
\end{array}\right)\vec{\psi}+\sum_{_{j=1}}^{v}V_{j}\left(\cdot-\vec{v_{j}}t\right)\vec{\psi}=0,\,\,\,\vec{\psi}|_{t=0}=\vec{\psi}_{0}
\]
where $\vec{v_{j}}$ are distinct vectors in $\mathbb{R}^{3}$, and
$V_{j}$ are matrix potentials of the form 
\[
V_{j}(t,x)=\left(\begin{array}{cc}
U_{j}(x) & -e^{i\theta_{j}(t,x)}W_{j}(x)\\
e^{-i\theta_{j}(t,x)}W_{j}(x) & -U_{j}(x)
\end{array}\right),
\]
where $\theta_{j}(t,x)=\left(\left|\vec{v_{j}}\right|^{2}+\alpha_{j}^{2}\right)t+2x\cdot\vec{v_{j}}+\gamma_{j}$
with $\alpha_{j}, \gamma_{j}\in\mathbb{R}$ and $\alpha_{j}\neq0$.
Furthermore, we require that each 
\[
H_{j}=\left(\begin{array}{cc}
-\frac{1}{2}\Delta+\frac{1}{2}\alpha_{j}^{2}+U_{j} & -W_{j}\\
W_{j} & \frac{1}{2}\Delta-\frac{1}{2}\alpha_{j}^{2}-U_{j}
\end{array}\right)
\]
satisfies the admissible conditions (Definition\ref{Admissible})
and stability condition (Definition \ref{Stability}) defined below.
\end{defn}
Here we give the definitions of stability condition and admissible
conditions for a matrix Hamiltonian $A=B+V$ where 
\[
B=\left(\begin{array}{cc}
-\frac{1}{2}\Delta+\mu & 0\\
0 & \frac{1}{2}\Delta-\mu
\end{array}\right),\,\,\, V=\left(\begin{array}{cc}
U & -W\\
W & -U
\end{array}\right)
\]
with $\mu>0$ and $U,\, W$ are of real-valued.
\begin{defn}
\label{Admissible}Let $A=B+V$ as above with $V$ exponentially decaying.
We call the operator $A$ on $\mathcal{H}:=L^{2}\left(\mathbb{R}^{3}\right)\times L^{2}\left(\mathbb{R}^{3}\right)$
admissible provided the following hold:
\end{defn}
1. $spec(A)\subset\mathbb{R}$ and $spec(A)\cap(-\mu,\mu)=\left\{ \omega_{\ell}:\,0\leq\ell\leq M\right\} $,
where $\omega_{0}=0$ and all $\omega_{j}$ are distinct eigenvalues.
There are no eigenvalues in $spec_{ess}(A)$.

2. For $1\leq\ell\leq M$, $L_{\ell}:=ker\left(A-\omega_{\ell}\right)^{2}=ker\left(A-\omega_{\ell}\right)$
and $ker\left(A\right)\subsetneq ker\left(A^{2}\right)=ker\left(A^{3}\right)=:L_{0}$.
Moreover, these spaces are finite-dimensional.

3. The ranges $Ran\left(A-\omega_{\ell}\right)$, for $1\leq\ell\leq M$
and $Ran\left(A^{2}\right)$ are closed.

4. The spaces $L_{\ell}$ are spanned by exponentially decreasing
functions in $\mathcal{H}$ (say, with bound $e^{-\epsilon_{0}\left|x\right|}$).

5. All these assumptions hold as well for the adjoint $A^{*}$. We
denote the corresponding (generalized) eigenspaces by $L_{\ell}^{*}$.

6. The points $\pm\mu$ are not resonances of $A$.
\begin{rem*}
For detailed definition of resonance here, one can find it in \cite{RSS}
Remark 7.10.
\end{rem*}
Following the above admissible conditions for $A$, we have can define
analogous projections onto continuous spectrum and point spectrum
following \cite{RSS} Lemma 7.3.
\begin{lem*}[\cite{RSS}, Lemma 7.3] 
	There a direct sum decomposition

\[
\mathcal{H}=\sum_{j=1}^{M}L_{j}+\left(\sum_{j=1}^{M}L_{j}^{*}\right)^{\perp}.
\]
The decomposition is invariant under $A$. Let $P_{c}$ denote the
projection onto $\left(\sum_{j=1}^{M}L_{j}^{*}\right)^{\perp}$ and
set $P_{b}=Id-P_{c}$. Notice that here $P_{c}$ is not an orthogonal projection.  It is easy to see $AP_{c}=P_{c}A$, and there exist numbers
$c_{ij}$ such that 
\[
P_{b}=\sum_{i,j}\phi_{j}c_{ij}\left\langle f,\psi_{i}\right\rangle ,\,\,\forall f\in\mathcal{H}
\]
 where $\phi_{j}$ and $\psi_{i}$ are exponentially decreasing functions.\end{lem*}
\begin{defn}
\label{Stability}For $A$ satisfying the admissible conditions, we
say $A$ satisfies the stability condition if 
\[
\mbox{\ensuremath{\sup}}_{t\in\mathbb{R}}\left\Vert e^{itA}P_{c}\right\Vert_{\mathcal{H}\rightarrow\mathcal{H}} <\infty.
\]

\end{defn}
In order the study the matrix charge transfer model, we need the vector-valued
Galilei transformation similarly as in the scalar case:
\[
\mathcal{G}_{\vec{v},y}(t)\left(\begin{array}{c}
\psi_{1}\\
\psi_{2}
\end{array}\right):=\left(\begin{array}{c}
\mathfrak{g}_{\vec{v},y}(t)\psi_{1}\\
\overline{\mathfrak{g}_{\vec{v},y}(t)\overline{\psi_{2}}}
\end{array}\right),
\]
where $\mathfrak{g}_{\vec{v},y}(t)$ is the scalar version Galilei
transformation. In contrast to the scalar case, the conjugated transformation
now involves a modulation $\mathcal{M}(t)$. We cite Lemma 8.2 in
\cite{RSS}.
\begin{lem*}
[\cite{RSS}, Lemma 8.2] Let $\alpha\in\mathbb{R}$ and let 
\[
A:=\left(\begin{array}{cc}
-\frac{1}{2}\Delta+\frac{1}{2}\alpha^{2}+U & -W\\
W & \frac{1}{2}\Delta-\frac{1}{2}\alpha^{2}-U
\end{array}\right)
\]
with real-valued $U$ and $W$. Moreover, let $\vec{v}\in\mathbb{R}^{3}$,
$\theta(t,x)=\left(\left|\vec{v}\right|^{2}+\alpha^{2}\right)t+2x\cdot\vec{v}+\gamma\,,\gamma\in\mathbb{R}$,
and define 
\[
H(t):=\left(\begin{array}{cc}
-\frac{1}{2}\Delta+U\left(\cdot-\vec{v}t\right) & -e^{i\theta\left(t,\cdot-\vec{v}t\right)}W\left(\cdot-\vec{v}t\right)\\
e^{-i\theta\left(t,\cdot-\vec{v}t\right)}W\left(\cdot-\vec{v}t\right) & \frac{1}{2}\Delta-\frac{1}{2}\alpha^{2}-U\left(\cdot-\vec{v}t\right)
\end{array}\right).
\]
Let $S(0)=Id$, $S(t)$ denote the propagator of the system 
\[
\frac{1}{i}\partial_{t}S(t)+H(t)S(t)=0.
\]
Finally, let 
\[
\mathcal{M}(t)=\mathcal{M}_{\alpha,\gamma}(t)=\left(\begin{array}{cc}
e^{-\frac{i\omega(t)}{2}} & 0\\
0 & e^{\frac{i\omega(t)}{2}}
\end{array}\right)
\]
where $\omega(t)=\alpha^{2}t+\gamma$. Then we have the following
relation 
\[
S(t)=\mathcal{G}_{\vec{v}}(t)^{-1}\mathcal{M}(t)^{-1}e^{-itA}\mathcal{M}(0)\mathcal{G}_{\vec{v}}(0).
\]

\end{lem*}
For matrix charge transfer models, the analysis should be similar
to the scalar case except that we have to modify the asymptotic orthogonality
condition. Recall that as we remarked above, it is not necessary to use the asymptotic
completeness results. In the scalar case, the asymptotic orthogonality condition is sufficient
for us. In the matrix case, the asymptotic orthogonality condition
is replaced by the definition of ``scattering states'' in Definition
8.3 in \cite{RSS} which is similar to the scattering space in the sense of Definition \ref{SP} for the scalar case.
\begin{defn}
\label{scattm}Let $U(t)\vec{\psi_{0}}=\vec{\psi}(t,\cdot)$, we call
that $\vec{\psi_{0}}$ a scattering state relative to $H_{j}$ if
\[
\left\Vert P_{b}\left(H_{j},t\right)U(t)\vec{\psi_{0}}\right\Vert _{L^{2}}\rightarrow0,\,\,\, t\rightarrow\infty.
\]
Here
\[
P_{b}\left(H_{j},t\right):=\mathcal{G}_{\vec{v}_{j}}(t)^{-1}\mbox{\ensuremath{\mathcal{M}}}_{j}(t)^{-1}P_{b}\left(H_{j}\right)\mbox{\ensuremath{\mathcal{M}}}_{j}(t)\mathcal{G}_{\vec{v}_{j}}(t)
\]
with $\mbox{\ensuremath{\mathcal{M}}}_{j}(t)=\mbox{\ensuremath{\mathcal{M}}}_{\alpha_{j},\gamma_{j}}(t)$.
\end{defn}
By the discussion in Section 8.3 in \cite{RSS}, if  $\vec{\psi_{0}}$ a scattering state relative to each $H_{j}$, we have the rate
of convergence similar to the scalar case, 
\[
\left\Vert P_{b}\left(H_{1},t\right)U(t)\vec{\psi_{0}}\right\Vert _{L^{2}}+\left\Vert P_{b}\left(H_{2},t\right)U(t)\vec{\psi_{0}}\right\Vert _{L^{2}}\lesssim e^{-\alpha t}\left\Vert \vec{\psi_{0}}\right\Vert _{L^{2}}
\]
 for some $\alpha>0$.

With all the preparations above, we now can formulate our Strichartz
estimates for matrix charge transfer models.
\begin{thm}
\label{thm:Strich-maxtrix}Consider the matrix charge transfer model
as in Definition \ref{Matrix}. We denote $\vec{\psi}(t)=U(t,0)\vec{\psi}_{0}$
and assume $\vec{\psi}_{0}$ is a scattering state relative to each
$H_{j}$ in sense of Definition \ref{scattm}. Then for a Schr\"odinger
admissible pair $(p,q)$ in $\mathbb{R}^{3}$, i.e., 
\begin{equation}
\frac{2}{p}+\frac{3}{q}=\frac{3}{2}\label{eq:schroadm}
\end{equation}
with $2\leq q\leq\infty,\,p\geq2$, we have 
\begin{equation}
\left\Vert \vec{\psi}\right\Vert _{L_{t}^{p}\left([0,\infty),\,L_{x}^{q}\right)}\leq C\left\Vert \vec{\psi}_{0}\right\Vert _{L_{x}^{2}}.\label{eq:Strim}
\end{equation}
 for some finite constant $C$.
\end{thm}
As in the scalar case, the proof Theorem \ref{thm:Strich-maxtrix} is based
on certain weighted estimates which rely on a bootstrap argument. Since
the proof is basically identical as with the scalar case, we do not
carry out the details. We only discuss it briefly. Recall that in
our proof, there are several important ingredients: dispersive estimates
for stationary potentials, the boundedness of wave operators, the
Kato smoothing estimate. All of them hold for the matrix case. For
the dispersive estimates for stationary potentials, one can find details
in \cite{Cu,RSS,ES}; for the boundedness of wave operators, the
results are discussed in \cite{Cu}; the Kato smoothing estimates
can be obtained as for the scalar case in \cite{RSS}.
Hence with the remark at the beginning of the second section, and
all the proofs above, we can conclude that Strichartz estimates
hold for the matrix case.

\begin{rem*}
	With the dispersive estimate for matrix transfer models and the results on scattering states, we can follow the proof in \cite{RSS} to prove the asymptotic completeness for matrix charge transfer Hamiltonians.
\end{rem*}

Similar to the scalar case, we also have the energy estimate.
\begin{thm}
\label{thm:energym}For $\vec{\psi}_{0}\in\mathcal{H}^{1}:=H^{1}\left(\mathbb{R}^{3}\right)\times H^{1}\left(\mathbb{R}^{3}\right)$,
we have 
\begin{equation}
\sup_{t\in\mathbb{R}}\left\Vert U(t,0)\vec{\psi}_{0}\right\Vert _{\mathcal{H}^{1}}\leq C\left\Vert \psi_{0}\right\Vert _{\mathcal{H}^{1}}.\label{416}
\end{equation}
\end{thm}
\begin{cor}
For $\vec{\psi}_{0}\in\mathcal{H}^{k}:=H^{k}\left(\mathbb{R}^{3}\right)\times H^{k}\left(\mathbb{R}^{3}\right)$
where $k$ is a non-negative integer, then we have 
\begin{equation}
\sup_{t\in\mathbb{R}}\left\Vert U(t,0)\psi_{0}\right\Vert _{\mathcal{H}^{k}}\leq C\left\Vert \psi_{0}\right\Vert _{\mathcal{H}^{k}}.\label{eq:417}
\end{equation}
\end{cor}

\end{document}